\documentclass[a4paper,reqno, 10pt]{amsart}

\usepackage{amsmath,amssymb,amsfonts,amsthm,mathrsfs}

\usepackage{verbatim,wasysym,cite}

\usepackage[latin1]{inputenc}
\usepackage{microtype}
\usepackage{color,enumitem,graphicx}
\usepackage[colorlinks=true,urlcolor=blue, citecolor=red,linkcolor=blue,
linktocpage,pdfpagelabels, bookmarksnumbered,bookmarksopen]{hyperref}
\usepackage[english]{babel}
\usepackage[symbol]{footmisc}
\renewcommand{\epsilon}{{\varepsilon}}

\numberwithin{equation}{section}
\newtheorem{theorem}{Theorem}[section]
\newtheorem{lemma}[theorem]{Lemma}
\newtheorem{remark}[theorem]{Remark}
\newtheorem{definition}[theorem]{Definition}
\newtheorem{proposition}[theorem]{Proposition}
\newtheorem{corollary}[theorem]{Corollary}

\def\({\left(}
\def\){\right)}
\def\<{\left\langle}
\def\>{\right\rangle}

\begin{document}

\title[NLS with rotation]
{Global well-posedness, blow-up and stability of standing waves for supercritical NLS with rotation}

\author[Alex H. Ardila]{Alex H. Ardila}
\address{Univ. Federal de Minas Gerais\\ ICEx-UFMG\\ CEP
  30123-970\\ MG, Brazil} 
\email{ardila@impa.br}
\author[Hichem Hajaiej]{Hichem Hajaiej}
\address{Department of Mathematics, California State University Los Angeles\\ 5151 University Drive, \\ Los Angeles}
\email{hhajaie@calstatela.edu}

\begin{abstract}
We consider the focusing mass supercritical nonlinear Schr\"odinger equation with rotation
\begin{equation*}
 iu_{t}=-\frac{1}{2}\Delta u+\frac{1}{2}V(x)u-|u|^{p-1}u+L_{\Omega}u,\quad (x,t)\in \mathbb{R}^{N}\times\mathbb{R},
\end{equation*}
where $N=2$ or $3$ and $V(x)$ is an anisotropic harmonic potential. Here $L_{\Omega}$ is the quantum mechanical angular momentum operator.  We establish conditions for global existence and blow-up in the energy space. Moreover, we prove strong instability of standing waves under certain conditions on the rotation and the frequency of the wave. Finally, we construct
orbitally stable standing waves solutions by considering a suitable local minimization problem. Those results are obtained  for nonlinearities which are $L^{2}$-supercritical. 
\end{abstract}


\maketitle

\medskip

\section{Introduction}

Consider the focusing nonlinear Schr\"odinger equation with rotation
\begin{equation}\label{GP}
 \begin{cases} 
iu_{t}=-\frac{1}{2}\Delta u+\frac{1}{2}V(x)u-|u|^{p-1}u+L_{\Omega}u,\quad (x,t)\in \mathbb{R}^{N}\times\mathbb{R},\\
u(x,0)=u_{0}(x),
\end{cases} 
\end{equation}
where $N=2$ or $3$, $u:\mathbb{R}^{N}\times\mathbb{R}\rightarrow \mathbb{C}$ and $1<p<2^{\ast}$. Here $2^{\ast}$ is defined by ${2}^{\ast}=1+\frac{4}{N-2}$ if $N= 3$, and $2^{\ast}=\infty$ if $N=2$. The  potential $V(x)$ is assumed to be harmonic,
\[ V(x)=\sum^{N}_{j=1}\gamma^{2}_{j}x^{2}_{j}, \quad x=(x_{1}, \ldots, x_{N})\in\mathbb{R}^{N}, \quad \gamma_{j}\in \mathbb{R}\setminus\left\{0\right\}.\]
The parameters $\gamma_{j}$ represent the harmonic trapping frequencies in each spatial direction. Through this paper we will assume that $\gamma:=\min_{1\leq j\leq N}\left\{\gamma_{j}\right\}>0$. The quantum mechanical angular momentum operator
$L_{\Omega}$ is expressed by $L_{\Omega}:=-\Omega \cdot L$, $L:=-ix \wedge \nabla$, where  $\Omega\in \mathbb{R}^{3}$ is the angular velocity vector. Notice that in $N=2$ the angular momentum operator takes the form:
\[ L_{\Omega}=-i|\Omega|(x_{1}\partial_{x_{2}}-x_{2}\partial_{x_{1}}),\]
where $\Omega=(0,0, |\Omega|)\in \mathbb{R}^{3}$.  When the angular momentum operator $L_{\Omega}=0$, Eq. \eqref{GP}  is known as a model to describe the Bose-Einstein condensate under a magnetic trap. We refer the readers to \cite{ACM2020, Fukui2000, OhSi, ZEc} for more information. If $L_{\Omega}\neq 0$, the model equation \eqref{GP} describes the Bose Einstein condensate  with rotation, which appears in a variety of physical settings such as the description of nonlinear waves and propagation of a laser beam in the optical fiber \cite{AFP1, RSP2}. 
We refer the readers to \cite{LiebSeiring2006} for a rigorous derivation in the stationary case of \eqref{GP}. Recently, the equation \eqref{GP} has attracted attentions due to their significance in theory and applications, see \cite{AEWP, ANS2019, RM2006, AntoineTangZhang2016, BaoCai2013, Guo2011, BHHZ} and references therein. Antonelli et al. in \cite{AEWP} proved existence and uniqueness of the Cauchy problem. Moreover, they also showed the existence of blow-up solutions in the $L^{2}$-critical and supercritical case (see also \cite{ANS2019}).  The issue of stability of standing waves in the $L^{2}$-subcritical case have been investigated in \cite{ANS2019}.

Note that we can rewrite the equation \eqref{GP} as
\begin{equation*}
iu_{t}=\frac{1}{2}R_{\Omega}u-|u|^{p-1}u,
\end{equation*}
where the operator $R_{\Omega}:=-\Delta +V(x)+2L_{\Omega}$ admit a precise interpretation as self-adjoint operator on $L^{2}(\mathbb{R}^{N})$ associated with the quadratic form (see \cite[Proposition 3.1]{ARE})
\begin{equation*}
\mathfrak{t}[u]:=\|\nabla u\|_{2}^{2}+\int_{\mathbb{R}^{N}}V(x)|u(x)|^{2}dx+2l_{\Omega}(u)
\end{equation*}
defined on the domain 
\begin{equation*} 
\mathrm{dom}(\mathfrak{t})=\Sigma:=\left\{u\in H^{1}(\mathbb{R}^{N}):|x|^{}u\in L^{2}(\mathbb{R}^{N}) \right\}.
\end{equation*}
Here $l_{\Omega}(u):=\left\langle L_{\Omega}u, u\right\rangle$ is the angular momentum. We observe that an integration by parts shows that the angular momentum  $l_{\Omega}(u)$ is always real valued. Formally, the NLS \eqref{GP} has the following two conserved quantities. The first conserved quantity is the energy 
\begin{equation*}
E_{\Omega}(u)=\frac{1}{2}\mathfrak{t}[u]-\frac{2}{p+1}\int_{\mathbb{R}^{N}}|u|^{p+1}dx.
\end{equation*}
The other conserved quantity is the mass
\begin{equation*}
M(u)=\|u\|^{2}_{2}.
\end{equation*}
Notice that due to the appearance of the angular momentum term, the energy functional $E_{\Omega}$ fails to be finite as well of class $C^{1}$ on $H^{1}(\mathbb{R}^{N})$ (even when the potential $V(x)$ is chosen to be identically zero). The local well-posedness for the Cauchy problem \eqref{GP} in the energy space $\Sigma$, equipped with the norm 
\begin{equation*}
\|u\|^{2}_{\Sigma}=\int_{\mathbb{R}^{N}}\left(\left|\nabla u\right|^{2}+|x|^{2}|u|^{2}+|u|^{2}\right)dx,
\end{equation*}
 can be proved using  Strichartz estimates \cite[Theorem 2.2]{AEWP}. More precisely, we have the following result.
\begin{proposition}
Let $u_{0}\in \Sigma$. Then there exists $T_{+}\in(0, \infty]$ and a unique maximal solution $u\in C([0,T_{+}), \Sigma)$ of the Cauchy problem \eqref{GP} with $u(0)=u_{0}$.  If $T_{+}=\infty$, then  $u$ is called a global solution in positive time. If $T_{+}<\infty$, then 
\begin{equation*}
\lim_{t\rightarrow T_{+}} \|\nabla u(t)\|^{2}_{2}=\infty 
\end{equation*}
and $u$ is called blows up in positive time.  Moreover, the solution enjoys the conservation of energy and mass  i.e.,
\begin{equation}\label{Cls}
E_{\Omega}(u(t))=E_{\Omega}(u_{0}), \quad M(u(t))=M(u_{0}) \quad \text{for every $t\in [0,T_{+})$.}
\end{equation}
\end{proposition}
We note that the evolution of the angular momentum under the flow generated by \eqref{GP} is given by (see \cite[Theorem 2.1.]{AEWP})
\begin{equation}\label{Amf}
l_{\Omega}(u(t))=l_{\Omega}(u_{0})+\int^{t}_{0}\int_{\mathbb{R}^{N}}i|u(x,t)|^{2}(\Omega\cdot L)V(x)dx, \quad t\in [0, T_{+}).
\end{equation}
By using a time-dependent change of coordinates and the conservation laws \eqref{Cls}, we have the global existence of Cauchy problem \eqref{GP}  in the $L^{2}$-subcritical case $1<p<1+\frac{4}{N}$ (see \cite[Theorem 2.2]{AEWP} for more details). As observed in \cite{AEWP}, we have that in the $L^{2}$-supercritical case $1+\frac{4}{N}<p<2^{\ast}$ blow-up of the solution may occur. In the super-critical case, the sharp thresholds of blow-up and global existence become very interesting.  In our first result, we establish sufficient and necessary conditions of global existence and  blow-up in finite time for the rotational NLS \eqref{GP} in the  mass supercritical regime. 
\begin{remark}
If the trapping frequencies are equal in each spatial direction, i.e., $\gamma=\gamma_{j}$ for all $j=1$, $\ldots$, $N$, then we also have the conservation of the angular momentum $ l_{\Omega}(u(t))=l_{\Omega}(u_{0})$ for every $t\in [0,T_{+})$. In particular, since we have the conservation of the angular momentum, it is not difficult to show that the condition $\|u_{0}\|_{2}<\|Q\|_{2}$ is sharp for global existence in the $L^{2}$-critical case $p=1+\frac{4}{N}$, where $Q$ is the unique positive and radially symmetric solution of 
\begin{equation}\label{Izx}
-\frac{1}{2}\Delta Q+Q-|Q|^{p-1}Q=0 \quad\text{in} \quad \mathbb{R}^{N}
\end{equation}
with $p=1+\frac{4}{N}$. 
\end{remark}
It is  convenient to introduce the number $s_{c}$ defined as
\begin{equation*}
s_{c}:=\frac{N}{2}-\frac{2}{p-1}.
\end{equation*}
Notice that $0\leq s_{c}<1$ if and only if $1+\frac{4}{N}\leq p<2^{\ast}$. We refer to the cases $s_{c}=0$ and $0<s_{c}<1$ as mass critical regime and mass supercritical regime, respectively. 

If $u(t)$ is the corresponding solution of the Cauchy problem \eqref{GP} with $u(0)=u_{0}$, we set 
\begin{equation}\label{ml}
l:=\inf_{t\in [0, T_{+})}  l_{\Omega}(u(t))\in [-\infty, \infty).
\end{equation}
\begin{remark}\label{Roq}
(i) As mentioned above, if the trapping frequencies are equal in each spatial direction, then we have that $l_{\Omega}(u(t))=l_{\Omega}(u_{0})$ for every $t\in [0,T_{+})$. This implies that in this case $l=l_{\Omega}(u_{0})\in \mathbb{R}$.\\
(ii) Notice that if  the nonlinearity is $L^{2}$-subcritical ($p<1+4/N$), then $l \in  \mathbb{R}$. Indeed, by \cite[Theorem 2.1]{AEWP} we see that if $u(t)$ is the solution of \eqref{GP}, $u(t)$ exists globally and  there exits $C>0$ such that $\|xu(t)\|^{2}_{L^{2}}+\|\nabla u(t)\|^{2}_{L^{2}}\leq C$ for all $t\in \mathbb{R}$. This implies by inequality \eqref{Emii} below that $|l_{\Omega}(u(t))|$ is uniformly bounded. Therefore, $l \in  \mathbb{R}$.\\
(iii) In the $L^{2}$-supercritical case ($p>1+4/N$), if $|\Omega|<\gamma$ and $\|u_{0}\|_{\Sigma}$ is small enough, a standard argument shows that there exists $C>0$ such that $\|u(t)\|_{\Sigma}\leq C$ for every $t$ in the interval of existence. Thus, we can apply the local theory to extend the solution such that  $\|u(t)\|_{\Sigma}\leq C$ for every $t\in \mathbb{R}$. Again,  by inequality \eqref{Emii} below we infer that $l \in  \mathbb{R}$.
\end{remark}

For $p>1+4/N$ (i.e. $0< s_{c}<1$) and $u_{0}\in \Sigma$, if $l \in  \mathbb{R}$ and $E_{\Omega}(u_{0})\geq l$, we define the following subsets in $\Sigma$, 
\begin{gather*}
\mathcal{K}^{+}=\bigl\{u_{0}\in \Sigma:\left(E_{\Omega}(u_{0})-l\right)^{s_{c}}M(u_{0})^{1-s_{c}}<E_{0,0}(Q)^{s_{c}}M(Q)^{1-s_{c}}\\
\|\nabla u_{0}\|^{s_{c}}_{2}\|u_{0}\|^{1-s_{c}}_{2}<\|\nabla Q\|^{s_{c}}_{2}\| Q\|^{1-s_{c}}_{2}\bigl\},
\end{gather*}
and
\begin{gather*}
 \mathcal{K}^{-}=\bigl\{u_{0}\in \Sigma: \left(E_{\Omega}(u_{0})-l\right)^{s_{c}}M(u_{0})^{1-s_{c}}<E_{0,0}(Q)^{s_{c}}M(Q)^{1-s_{c}}\\
 \|\nabla u_{0}\|^{s_{c}}_{2}\| u_{0}\|^{1-s_{c}}_{2}>\|\nabla Q\|^{s_{c}}_{2}\| Q\|^{1-s_{c}}_{2}\bigl\},
\end{gather*}
where $Q$ denotes the unique positive and radially symmetric solution of  \eqref{Izx} and $E_{0,0}(Q)=\frac{1}{2}\|\nabla Q\|^{2}_{2}-\frac{2}{p+1}\|Q\|^{p+1}_{p+1}$. Notice that $\mathcal{K}^{\pm}\neq \emptyset$ (see Remark \ref{empty} below).

In our first result, we obtain a criteria between blow-up and global existence for \eqref{GP} in terms of the energy, mass and $l$ given by
\eqref{ml}.
\begin{theorem}\label{SC}
Let $1+\frac{4}{N}<p<2^{\ast}$ $(i.e., 0<s_{c}<1)$, $u_{0}\in \Sigma$ and let $u\in C([0, T_{+}), \Sigma)$ be the corresponding solution of \eqref{GP} with initial data $u_{0}$. \\
(i) If $l=-\infty$, then there exists a sequence of times $\left\{t_{n}\right\}$  such that $t_{n}\rightarrow T_{+}$ and  $\lim_{t_{n}\rightarrow T_{+}}\|\nabla u(t_{n})\|^{2}_{L^{2}}=\infty$.\\
(ii) Assume that $l\in \mathbb{R}$ and $E_{\Omega}(u_{0})\geq l$. Then one of the following two cases holds:
\begin{enumerate}
\item If $u_{0}\in \mathcal{K}^{+}$, then  the corresponding solution $u(t)$ exists globally.
\item If  $u_{0}\in \mathcal{K}^{-}$, the  solution blows-up in finite time.
\end{enumerate}
Moreover, the sets $\mathcal{K}^{\pm}$ are invariant by the flow of the equation \eqref{GP}.\\
(iii) Assume that $l\in\mathbb{R}$ and $E_{\Omega}(u_{0})<l$. Then the solution $u(t)$ blows up at finite time in $\Sigma$. In addition, for every $t$ in the existence time we have 
\begin{equation*}
\|\nabla u(t)\|_{2}\geq \left(\frac{(p-1)N}{4}\right)^{\frac{1}{s_{c}(p-1)}}\left(\frac{\| Q\|_{2}}{\| u_{0}\|_{2}}\right)^{\frac{1-s_{c}}{s_{c}}}\|\nabla Q\|_{2}.
\end{equation*}
\end{theorem}

For the standard Schr\"odinger equation, the sharp thresholds of global existence and  blow-up have  been extensively studied during the past decades (see \cite{DR5,HR1,CB} and references therein). To prove the Theorem \ref{SC} we follow the arguments developed in
Holmer and Roudenko  \cite{HR1, HR23}, where they proved similar results for the $L^{2}$-supercritical NLS with zero potential.  

\begin{remark}\label{empty}
(i) The set $\mathcal{K}^{+}$ is not empty for $|\Omega|<\gamma$. Indeed, if $\|u_{0}\|_{\Sigma}$ is small enough, by  Remark \ref{Roq} (iii) we have that $l \in  \mathbb{R}$. Moreover, by the energy conservation and the  Gagliardo-Nirenberg inequality (see \eqref{GI}) we see that
\[ E(u_{0})-l_{\Omega}(u(t)) \geq X(t)-CX(t)^{N(p-1)/4},\]
where $C>0$ and $X(t)=\frac{1}{2}\|\nabla u(t)\|^{2}_{2}+\frac{1}{2}\int_{\mathbb{R}^{N}}V(x)|u(x,t)|^{2}dx$. Since  $p>1+4/N$, taking $\|u_{0}\|_{\Sigma}$ 
is small enough we infer that $E(u_{0})-l_{\Omega}(u(t))\geq 0$. This implies that  $E(u_{0})-l\geq 0$. In conclusion, there exists $\epsilon>0$
such that if  $\|u_{0}\|_{\Sigma}<\epsilon$, then $u_{0}\in \mathcal{K}^{+}$.\\
(ii)We can extend the Theorem \ref{SC} to the case of potentials $V\in C^{\infty}(\mathbb{R}^{N})$  such that $V\geq 0$ and $\partial^{\alpha}V\in L^{\infty}(\mathbb{R}^{N})$ for all multi-indices $\alpha\in \mathbb{N}^{N}$ with $|\alpha|\leq 2$. Indeed, the proof of Theorem \ref{SC} works after obvious modifications. Notice also that in this case if $(\Omega\cdot L)V(x)\geq 0$ (see \eqref{Amf}), then  we have that $l_{\Omega}(u(t))\geq l_{\Omega}(u_{0})$, for all $t\in [0, T_{+})$, i.e., 
$l=l_{\Omega}(u_{0})$. As a consequence of this fact, we see that if $E_{\Omega}(u_{0})-l_{\Omega}(u_{0})$ is small enough and $\|u_{0}\|^{2}_{L^{2}}$ is sufficiently large, then $u_{0}\in \mathcal{K}^{-}$. Similarly, if $\|u_{0}\|^{2}_{L^{2}}$  is small enough then $u_{0}\in \mathcal{K}^{+}$.
\end{remark}

\begin{remark}
(i) Notice that if $l=-\infty$, then by Theorem \ref{SC} one of the following two statements is true:
\begin{enumerate}
\item The solution blows-up in finite time, i.e, $T_{+}<\infty$ and $\lim_{t\rightarrow T_{+}}\|\nabla u(t)\|^{2}_{{2}}=\infty$.
\item The solution grows-up in time, i.e, $T_{+}=\infty$  and there exists a sequence $t_{n}\rightarrow\infty$ such that $\lim_{n\rightarrow \infty}\|\nabla u(t_{n})\|^{2}_{2}=\infty$.
\end{enumerate}
(ii) We observe that in the mass supercritical regime  $1+\frac{4}{N}< p<2^{\ast}$, if $E_{\Omega}(u_{0})\geq l$, then the condition $\|\nabla u_{0}\|^{s_{c}}_{2}\| u_{0}\|^{1-s_{c}}_{2}<\|\nabla Q\|^{s_{c}}_{2}\| Q\|^{1-s_{c}}_{2}$ is sharp for global existence except for the threshold level $\|\nabla u_{0}\|^{s_{c}}_{2}\| u_{0}\|^{1-s_{c}}_{2}=\|\nabla Q\|^{s_{c}}_{2}\| Q\|^{1-s_{c}}_{2}$.
\end{remark}

In the second part of this paper, we study the stability and instability of standing waves.  Throughout this paper, we call a standing wave a solution of \eqref{GP} with the form $u(x,t)=e^{\frac{\omega}{2}it}\varphi_{\omega}(x)$, where $\omega\in \mathbb{R}$ is a frequency and $\varphi_{\omega}$ satisfying the following nonlinear elliptic problem
\begin{equation}\label{EE}
 \begin{cases} 
 -\Delta \varphi+\omega\varphi+V(x)\varphi-2|\varphi|^{p-1}\varphi+2 L_{\Omega}\varphi=0,\\
\varphi\in \Sigma\setminus \left\{0\right\}.
\end{cases} 
\end{equation}
For $\gamma=\min_{1\leq j\leq N}\left\{\gamma_{j}\right\}>0$, it is well known that operator $R_{ \Omega}$ has a purely discrete spectrum (see \cite[Theorem 2.2]{MATSUNU} for more details). Thus, we define
\begin{equation}\label{impr}
\lambda_{0}:=-\inf\left\{\|\nabla u\|_{2}^{2}+\int_{\mathbb{R}^{N}}V(x)|u(x)|^{2}dx+2l_{\Omega}(u): u\in \Sigma, \| u\|_{L^{2}}^{2}=1 \right\}.
\end{equation}
Moreover, we define the following functionals of class $C^{2}$:
\begin{align*}
 S_{\omega}(u)&=\frac{1}{2}\mathfrak{t}[u]+\frac{\omega}{2}\int_{\mathbb{R}^{N}}|u|^{2}dx-\frac{2}{p+1}\int_{\mathbb{R}^{N}}|u|^{p+1}dx,\\
 I_{\omega}(u)&=\mathfrak{t}[u]+{\omega}\int_{\mathbb{R}^{N}}|u|^{2}dx-{2}\int_{\mathbb{R}^{N}}|u|^{p+1}dx,\\
P(u)&=\frac{1}{2}\int_{\mathbb{R}^{N}}|\nabla u|^{2}dx-\frac{1}{2}\int_{\mathbb{R}^{N}}V(x)|u(x)|^{2}dx-\frac{N(p-1)}{2(p+1)}\int_{\mathbb{R}^{N}}|u|^{p+1}dx.
\end{align*}
We observed that  the elliptic equation \eqref{EE} can be written as $S^{\prime}(\varphi)=0$. Now, for $\omega>\lambda_{0}$, we denote the set of non-trivial solutions of \eqref{EE} by
\begin{equation*}
 \mathcal{A}_{\omega}=\bigl\{ \varphi\in \Sigma\setminus  \left\{0 \right\}: S^{\prime}_{\omega}(\varphi)=0\bigl\}.
\end{equation*}
A ground states for \eqref{EE} is a function $\phi\in \mathcal{A}_{\omega}$  that minimizes $S_{\omega}$ over the set $\mathcal{A}_{\omega}$. The set of ground states is denoted by $\mathcal{G}_{\omega}$ and 
\begin{equation*}
 \mathcal{G}_{\omega}=\bigl\{ \varphi\in  \mathcal{A}_{\omega}: S_{\omega}(\varphi)\leq S_{\omega}(v)\quad \text{for all $v\in \mathcal{A}_{\omega}$}\bigl\}.
\end{equation*}
In the following result, we prove that the set of ground states $\mathcal{G}_{\omega}$ is not empty.
\begin{proposition} \label{GNe}
Let  $|\Omega|<\gamma$, $\omega>\lambda_{0}$ and $1<p<2^{\ast}$. Then the set of ground states $\mathcal{G}_{\omega}$ is not empty. Moreover, we have the following variational characterization
\begin{equation*}
 \mathcal{G}_{\omega}=\bigl\{ \varphi\in \Sigma: S_{\omega}(\varphi)=d(\omega)\quad \text{and}\quad I_{\omega}(u)=0\bigl\},
\end{equation*}
where
\begin{equation*}
d(\omega)={\inf}\left\{S_{\omega}(u):\, u\in \Sigma \setminus  \left\{0 \right\},  I_{\omega}(u)=0\right\}.
\end{equation*}
\end{proposition}
Next we need the following definition.
\begin{definition} 
 We say that the set $\mathcal{M}\subset \Sigma$ is $\Sigma$-stable under the flow generated by \eqref{GP} if, for $\epsilon>0$, there exists $\eta>0$ such that for any initial data $u_{0}$ satisfying
\begin{equation*}
\inf_{v\in \mathcal{M}}\|u_{0}-v\|_{\Sigma}<\delta,
\end{equation*}
then the corresponding solution $u(t)$ of \eqref{GP} with $u(0)=u_{0}$ exists for all $t\in \mathbb{R}$ and satisfies
\begin{equation*}
\inf_{v\in \mathcal{M}}\|u(t)-v\|_{\Sigma}<\epsilon.
\end{equation*}
Otherwise, $\mathcal{M}$ is said to be unstable. We say that the standing wave $u(x,t)=e^{\frac{\omega}{2}it}\varphi_{\omega}(x)$ of \eqref{GP} is stable in $\Sigma$ if $\mathcal{O}_{\omega}$ is stable and $u(x,t)=e^{\frac{\omega}{2}it}\varphi_{\omega}(x)$ is unstable if $\mathcal{O}_{\omega}$ is unstable, where $\mathcal{O}_{\omega}=\left\{e^{i\theta}\varphi_{\omega}: \theta\in \mathbb{R}\right\}$.
\end{definition}

Following the argument by Fukuizumi and Ohta \cite{FOIB2003},  we can show a sufficient condition for the instability of standing waves in the mass supercritical regime.
\begin{theorem}\label{Ax1}
Let  $|\Omega|<\gamma$, $\omega>\lambda_{0}$, $1+\frac{4}{N}<p<2^{\ast}$ and $\phi_{\omega}\in \mathcal{G}_{\omega}$. 
Assume that $\partial^{2}_{s}E_{\Omega}(\phi^{s}_{\omega})|_{s=1}<0$, where $\phi^{s}_{\omega}(x)=s^{\frac{N}{2}}\phi_{\omega}(sx)$.
Then the standing wave $e^{\frac{\omega}{2}i t}\phi_{\omega}(x)$ of \eqref{GP} is  unstable in $\Sigma$.
\end{theorem}

Under some conditions on the rotation $|\Omega|$ and frequency $\omega$, it is possible to show that $\partial^{2}_{s}E_{\Omega}(\phi^{s}_{\omega})|_{s=1}< 0$. Notice that, since the standing wave $e^{\frac{\omega}{2}i t}\phi_{\omega}(x)$ of \eqref{GP} with $\Omega=0$ is strongly unstable in $\Sigma$ when $p>1+\frac{4}{N}$ and $\omega$ is sufficiently large (see \cite{Fukui2000}), we expect that the standing wave solution $e^{\frac{\omega}{2}i t}\phi_{\omega}(x)$ of \eqref{GP} with  $|\Omega|\ll\gamma$ can also be unstable in $\Sigma$ when $p>1+\frac{4}{N}$ and $\omega$ is sufficiently large. Indeed,  we have the following result.
\begin{corollary}\label{Eis}
Let $1+\frac{4}{N}<p<2^{\ast}$ and $\phi_{\omega}\in \mathcal{G}_{\omega}$. There exists $\epsilon>0$ small enough such that if  $|\Omega|^{2}\leq\epsilon\gamma^{2}$, then there is a sequence  $\left\{\omega_{n}\right\}^{\infty}_{n=1}$ such that the standing wave $e^{\frac{\omega_{n}}{2} i t}\phi_{\omega_{n}}(x)$ of \eqref{GP} is unstable. Moreover,  $\omega_{n}\rightarrow \infty$ as $n\rightarrow\infty$.
\end{corollary}
\begin{remark}
We observe that under the conditions of Theorem \ref{Ax1}, if the trapping frequencies are equal ($\gamma=\gamma_{j}$, $j=1$, $\ldots$ $N$), then thanks to the conservation of the angular momentum it is possible to show that the standing wave $e^{\frac{\omega}{2}i t}\phi_{\omega}(x)$ of \eqref{GP} is strongly unstable in $\Sigma$. Indeed, the proof follows from exactly the same argument in Ohta \cite[Theorem 1]{OhSi}. In particular, we infer that the standing wave $e^{\frac{\omega_{n}}{2} i t}\phi_{\omega_{n}}(x)$ in Corollary \ref{Eis} is strongly unstable in $\Sigma$ (see the proof of Corollary \ref{Eis} and  Lemma \ref{Ls1} below). 
\end{remark}
Now, we focus on the stability of standing waves in the mass supercritical regimen $p>1+\frac{4}{N}$. The more common approach to construct orbitally stable standing waves to \eqref{GP} is to consider the following constrained minimization problems
\begin{equation*}
J_{q}=\inf\left\{E_{\Omega}(u),\quad u\in \Sigma,\quad \|u\|^{2}_{L^{2}}=q \right\}.
\end{equation*}
In the mass subcritical case $p<1+\frac{4}{N}$ it is possible to show that $J_{q}>-\infty$  and any minimizing sequence of $J_{q}$ is relatively compact in $\Sigma$ (see \cite{ANS2019}). In particular, this implies that  the set of minimizers of $J_{q}$ is $\Sigma$-stable under the flow generated by \eqref{GP}.

On the other hand, in the mass supercritical case $p>1+\frac{4}{N}$, we have  $J_{q}=-\infty$. Indeed, we set $u_{\mu}(x):=\mu^{\frac{N}{2}}u(\mu x)$ where  $u\in \Sigma$ with $\|u\|^{2}_{2}=q$. It is not difficult to show that $\|u_{\mu}\|^{2}_{2}=\|u\|^{2}_{2}$, $l_{\Omega}(u_{\mu})=l_{\Omega}(u)$ and 
\begin{equation*}
E_{\Omega}(u_{\mu})=\frac{\mu^{2}}{2}\|\nabla u\|^{2}_{2}+\mu^{-2}\frac{1}{2}\int_{\mathbb{R}^{N}}V(x)|u(x)|^{2}dx-\frac{2\mu^{\frac{N}{2}(p-1)}}{p+1}\|u\|^{p+1}_{p+1}+l_{\Omega}(u).
\end{equation*}
Since $p>1+\frac{4}{N}$, we infer that $E_{\Omega}(u_{\mu})\rightarrow -\infty$ as $\mu$ goes to $+\infty$, and therefore $J_{q}=-\infty$. To overcome this difficulty, we consider a local minimization problem. Following \cite{BEBOJEVI2017}, for $|\Omega|<\gamma$, we define the following subsets:
\begin{align*}
D_{q}:=&\left\{u\in\Sigma: \|u\|^{2}_{2}=q \right\}, \\ 
B_{r}:=&\left\{u\in\Sigma: \|u\|^{2}_{H}\leq r \right\}, 
\end{align*}
where $\|\cdot\|_{H}$ denotes the norm (see Section \ref{S:2})
\begin{equation*}
\|u\|^{2}_{H}:= \|\nabla u\|^{2}_{2}+\int_{\mathbb{R}^{N}}V(x)|u(x)|^{2}dx+2l_{\Omega}(u).
\end{equation*}
Moreover, for a fixed $q>0$ and $r>0$, we set the following local variational problem
\begin{equation}\label{Vp1}
J^{r}_{q}=\inf\left\{E_{\Omega}(u),\quad u\in D_{q}\cap B_{r}\right\}.
\end{equation}
Using the Gagliardo-Nirenberg inequality, it is not difficult to show that if $D_{q}\cap B_{r}\neq \emptyset$, then the variational problem $J^{r}_{q}$ is well defined; that is, $J^{r}_{q}>-\infty$ (see proof of Lemma \ref{Lxa90} below). Let us denote the set of nontrivial solutions of \eqref{Vp1} by
\begin{equation*}
\mathcal{G}^{r}_{q}:=\left\{v\in D_{q}\cap B_{r}: \quad \text{$v$ is a minimizer of \eqref{Vp1}}\right\}.
\end{equation*}
The following result shows that, in the mass supercritical regime, the set $\mathcal{G}^{r}_{q}$ is not empty.
\begin{theorem}\label{Th2}
Let  $|\Omega|<\gamma$ and $1+\frac{4}{N}<p<2^{\ast}$. For any $r>0$ there exists $q_{0}>0$ such that for all $q<q_{0}$ we have:\\
(i) Any minimizing sequence for \eqref{Vp1} is precompact in $\Sigma$. \\
(ii) For every $\varphi\in\mathcal{G}^{r}_{q}$ there exists a Lagrange multiplier $\omega\in \mathbb{R}$ such that the stationary problem  \eqref{EE} is satisfied with the estimates
\begin{equation*}
\lambda_{0}<\omega\leq \lambda_{0}(1-Cq^{\frac{p-1}{2}}).
\end{equation*}
\end{theorem}
Note that from the above theorem,  $\omega\rightarrow \lambda_{0}$ as $q\rightarrow 0$. Moreover, if $\varphi\in \mathcal{G}^{r}_{q}$, then there exists $\omega>\lambda_{0}$ such that $\varphi$ is a solution of stationary problem \eqref{EE}. In particular, $u(x,t)=e^{\frac{\omega}{2} i t}\varphi(x)$ is a standing wave solution to \eqref{GP}. \\
We have the following stability result for the set $\mathcal{G}^{r}_{q}$.
\begin{corollary}\label{Et}
If $|\Omega|<\gamma$,  then for any fixed $r>0$ and $q<q_{0}$ given in the Theorem \ref{Th2} we have that the set $\mathcal{G}^{r}_{q}$ is $\Sigma$-stable with respect to \eqref{GP}.
\end{corollary}
We remark that nothing is known about orbital stability of standing waves in the supercritical case when $|\Omega|>\gamma$. The study of  the stability seems highly non-trivial; see the discussion presented after formula (1.6) in \cite{ANS2019} for more details.

This paper is organized as follows. In Section \ref{S:1} we prove our global existence/blow-up result stated in Theorem \ref{SC}. In Section \ref{S:2}  we prove, by variational techniques, the existence of ground states (Proposition \ref{GNe}).  In Section \ref{S:4}, we analyze the  instability of the standing waves in Corollary \ref{Eis}.  Finally,  Section \ref{S:5} is devoted to the proof of Theorem \ref{Th2} and Corollary \ref{Et}.


\section{Conditions for Global existence and blow-up} \label{S:1}

In this section, we prove Theorem \ref{SC}. First we recall the sharp Gagliardo-Nirenberg inequality \cite{CB},
\begin{equation}\label{GI}
\|u\|^{p+1}_{{p+1}}\leq c_{GN}\|\nabla u\|^{\frac{N(p-1)}{2}}_{2}\|u\|^{p+1-\frac{N(p-1)}{2}}_{2},
\end{equation}
where the sharp constant $c_{GN}>0$ is explicitly given by
\begin{equation*}
c_{GN}=\left(\frac{2N(p-1)}{2(p+1)-N(p-1)}\right)^{\frac{4-N(p-1)}{4}}\frac{(p+1)}{N(p-1)\|Q\|^{p-1}_{2}}.
\end{equation*}
Next we recall the standard viral identity related to \eqref{GP} (see \cite{AEWP}). 
\begin{lemma}
Let $u_{0}\in \Sigma$ and $u(x,t)$ the corresponding solution of Cauchy problem \eqref{GP} on $[0,T)$, where $T$ is the maximum time of existence. We put $J(t):=\int_{\mathbb{R}^{N}}|x|^{2}|u(x,t)|^{2}\,dx$. Then we have for all $t\in [0,T)$
\begin{equation*}
J^{\prime}(t)=2Im\int_{\mathbb{R}^{N}}\left(\nabla u(x,t)\cdot x\right)\overline{u}(x,t)\,dx
\end{equation*}
and
\begin{equation*}
J^{\prime\prime}(t)=2\int_{\mathbb{R}^{N}}|\nabla u|^{2}\,dx-2\int_{\mathbb{R}^{N}}V(x)|u(x,t)|^{2}\,dx-2N\left(\frac{p-1}{p+1}\right)\int_{\mathbb{R}^{N}}|u(x,t)|^{p+1}\,dx.
\end{equation*}
\end{lemma}
Note that we can compute the virial identity in terms of $E_{\Omega}(u)$ and $l_{\Omega}(u)$. In indeed, a simple computation shows
\begin{align}\nonumber
J^{\prime\prime}(t)&=\left(\frac{4-N(p-1)}{2}\right)\|\nabla u(t)\|^{2}_{2}-\left(\frac{N(p-1)+4}{2}\right)\int_{\mathbb{R}^{N}}V(x)|u(x)|^{2}dx\\ 
&+N(p-1)\left(E_{\Omega}(u(t))-l_{\Omega}(u(t))\right). \label{VFM}
\end{align}
We will frequently use the following inequality 
\begin{equation}\label{Emii}
|l_{\Omega}(\psi)| \leq \frac{1}{2a}|\Omega|^{2}\|x\psi\|^{2}_{2}+\frac{a}{2}\|\nabla \psi\|^{2}_{2}.
\end{equation}

The proof of Theorem \ref{SC} is based on the following result.

\begin{lemma}\label{lemaf}
Under the conditions of the Theorem \ref{SC} the following statements hold. Assume that
\begin{equation}\label{C1}
\begin{split}
\left(E_{\Omega}(u_{0})-l\right)^{s_{c}}M(u_{0})^{1-s_{c}}&<E_{0,0}(Q)^{s_{c}}M(Q)^{1-s_{c}},\\
E_{\Omega}(u_{0})-l\geq 0,
\end{split}
\end{equation}
(i) If 
\begin{equation}\label{C2}
\|\nabla u_{0}\|^{s_{c}}_{2}\| u_{0}\|^{1-s_{c}}_{2}<\|\nabla Q\|^{s_{c}}_{2}\| Q\|^{1-s_{c}}_{2},
\end{equation}
then $u(t)$ is a global solution and for every $t\in \mathbb{R}$
\[
\|\nabla u(t)\|^{s_{c}}_{2}\| u_{0}\|^{1-s_{c}}_{2}<\|\nabla Q\|^{s_{c}}_{2}\| Q\|^{1-s_{c}}_{2}.
\]
(ii) If 
\begin{equation}\label{C4}
\|\nabla u_{0}\|^{s_{c}}_{2}\| u_{0}\|^{1-s_{c}}_{2}>\|\nabla Q\|^{s_{c}}_{2}\| Q\|^{1-s_{c}}_{2},
\end{equation}
 then the solution $u(t)$ blows up at finite time. Moreover, we also have
\begin{equation*}
\|\nabla u(t)\|^{s_{c}}_{2}\| u_{0}\|^{1-s_{c}}_{2}>\|\nabla Q\|^{s_{c}}_{2}\| Q\|^{1-s_{c}}_{2}
\end{equation*}
for every $t$ in the existence time.\\
(iii) If, in place of \eqref{C1} and \eqref{C4}, we assume 
\begin{equation}\label{Fe1}
 E_{\Omega}(u_{0})-l<0,
\end{equation}
then the solution $u(t)$ blows up at finite time in $\Sigma$. Moreover, for every $t$ in the existence time we have 
\begin{equation*}
\|\nabla u(t)\|_{2}\geq \left(\frac{(p-1)N}{4}\right)^{\frac{1}{s_{c}(p-1)}}\left(\frac{\| Q\|_{2}}{\| u_{0}\|_{2}}\right)^{\frac{1-s_{c}}{s_{c}}}\|\nabla Q\|_{2}.
\end{equation*}
\end{lemma}
\begin{proof}
The proof is inspired by the one of Theorem 2.1 in \cite{HR1} (see also \cite{HR23}). Let $u(t)$ be the corresponding solution of \eqref{GP} with initial data $u_{0}$.  By the sharp Gagliardo-Nirenberg inequality \eqref{GI} we get
\begin{align*}
E_{\Omega}(u)-l_{\Omega}(u)&=\frac{1}{2}\|\nabla u\|^{2}_{2}+\frac{1}{2}\int_{\mathbb{R}^{N}}V(x)|u(x)|^{2}dx-\frac{2}{p+1}\|u\|^{p+1}_{{p+1}} \\ 
& \geq \frac{1}{2}\|\nabla u\|^{2}_{2}-\frac{2\,c_{GN}}{p+1}\|\nabla u\|^{\frac{N(p-1)}{2}}_{2}\|u_{0}\|^{p+1-\frac{N(p-1)}{2}}_{2}.
\end{align*}
Define the function $f(x)=\frac{1}{2}x^{2}-\beta_{p,N} x^{\frac{N}{2}(p-1)}$, where $\beta_{p,N}=\frac{2\,c_{GN}}{p+1}\|u_{0}\|^{p+1-\frac{N(p-1)}{2}}_{2}$. Notice that $\text{deg}(f)\geq 2$ and 
\begin{align*}
f^{\prime}(x)&=x-\frac{N}{2}(p-1)\beta_{p,n} x^{\frac{N}{2}(p-1)-1}  \\ 
&=x\left(1-\frac{N}{2}(p-1)\beta_{p,N} x^{(p-1)s_{c}}\right).
\end{align*}
A simple computation shows that $f^{\prime}(x)=0$ when  $x_{0}=0$ and 
\begin{equation*}
x_{1}=\left(\frac{2}{N\beta_{p,N}(p-1)}\right)^{\frac{2}{N(p-1)-4}}=\left(\frac{p+1}{ N(p-1)c_{GN}}\right)^{\frac{1}{s_{c}(p-1)}}{\|u_{0}\|_{2}^{-\left(\frac{1-s_{c}}{s_{c}}\right)}}.
\end{equation*}
Notice that  $f$ has a local minimum at $x_{0}$ and a local maximum at $x_{1}$, with maximum value $f(x_{1})=\frac{s_{c}}{N}x^{2}_{1}$. 
Now, it is not difficult to show that
\begin{equation*}
\left(\frac{p+1}{ N(p-1)c_{GN}}\right)^{\frac{1}{s_{c}(p-1)}}=\left(\frac{2N(p-1)}{2(p+1)-N(p-1)}\right)^{\frac{1}{2}}{\|Q\|_{2}^{\frac{1}{s_{c}}}}.
\end{equation*}
Moreover, by the Pohozaev identities  we infer that
\begin{align*}
\|\nabla Q\|_{2}\| Q\|^{\frac{1-s_{c}}{s_{c}}}_{2}&=\left(\frac{2N(p-1)}{2(p+1)-N(p-1)}\right)^{\frac{1}{2}}{\|Q\|_{2}^{\frac{1}{s_{c}}}},\\
E_{0,0}(Q)M(Q)^{\frac{(1-s_{c})}{s_{c}}}&=\frac{s_{c}}{N}\left(\frac{2N(p-1)}{2(p+1)-N(p-1)}\right){\|Q\|_{2}^{\frac{2}{s_{c}}}}.
\end{align*}
In particular, since $f(x_{1})=\frac{s_{c}}{N}x^{2}_{1}$, using the condition \eqref{C1} we see that
\begin{align}\nonumber
E_{\Omega}(u_{0})-l&<E_{0,0}(Q)M(Q)^{\frac{(1-s_{c})}{s_{c}}}\|u_{0}\|^{-2\frac{(1-s_{c})}{s_{c}}}_{2}\\\nonumber
&= \frac{s_{c}}{N}\left(\frac{2N(p-1)}{2(p+1)-N(p-1)}\right){\|Q\|_{2}^{\frac{2}{s_{c}}}}\|u_{0}\|^{-2\frac{(1-s_{c})}{s_{c}}}_{2}\\ \label{IF1}
&=f(x_{1}).
\end{align}
Since $E_{\Omega}(u(t))$ is independent of $t$, we infer that
\begin{equation}\label{Ei}
f(\|\nabla u(t)\|_{2})\leq E_{\Omega}(u(t))-l_{\Omega}(u(t))\leq E_{\Omega}(u_{0})-l<f(x_{1}).
\end{equation}
On the other hand, using the condition \eqref{C2} we obtain
\begin{align}\nonumber
\|\nabla u_{0}\|_{2}&<\|\nabla Q\|_{2}\| Q\|^{\frac{(1-s_{c})}{s_{c}}}_{2}\|u_{0}\|^{-\frac{(1-s_{c})}{s_{c}}}_{2}\\ \nonumber
&=\left(\frac{2N(p-1)}{2(p+1)-N(p-1)}\right)^{\frac{1}{2}}{\|Q\|_{2}^{\frac{1}{s_{c}}}}\|u_{0}\|^{-\frac{(1-s_{c})}{s_{c}}}_{2}=x_{1}.
\end{align}
Therefore, by the continuity of $\|\nabla u(t)\|_{2}$ in $t$, and considering \eqref{Ei} we deduce that  
$\|\nabla u(t)\|_{2}<x_{1}$ for any $t$ as long as the solutions exists.  Denote by $I$ the maximal interval of existence of the solution $u$. Since  $\|\nabla u(t)\|_{2}\leq K$ for any $t\in I$, from the energy conservation and \eqref{GI} we get
\[\begin{split}
\frac{1}{2}\|\nabla u(t)\|^{2}_{2}+\frac{1}{2}\int_{\mathbb{R}^{N}}V(x)|u(x)|^{2}dx=E_{\Omega}(u_{0})-l_{\Omega}(u(t))+\frac{2}{p+1}\|u(t)\|^{p+1}_{p+1}\\ 
\leq E_{\Omega}(u_{0})-l+C\|\nabla u(t)\|^{\frac{N(p-1)}{2}}_{2}\|u(t)\|^{p+1-\frac{N(p-1)}{2}}_{2}\\
\leq E_{\Omega}(u_{0})-l+CK^{\frac{N(p-1)}{2}}\|u_{0}\|^{p+1-\frac{N(p-1)}{2}}_{2}.
\end{split}\]
for every $t\in I$. Thus, $\|u(t)\|^{2}_{\Sigma}$ is bounded for all time $t\in I$. Then we infer  that the solution exists globally in time. Next, we turn our attention to the proof of part (ii). Suppose by contradiction that the corresponding solution $u(t)$ of \eqref{GP} with $u(0)=u_{0}$ satisfies the hypothesis
\eqref{C1}-\eqref{C4} exists globally. Notice that by the condition \eqref{C4}, we have $\|\nabla u_{0}\|_{2}>x_{1}$. Now applying the condition \eqref{C1}, it is clear that there exists $\delta_{1}>0$ such that 
\begin{equation*}
(E_{\Omega}(u_{0})-l)M(u_{0})^{\frac{1-s_{c}}{s_{c}}}<(1-\delta_{1})E_{0,0}(Q)M(Q)^{\frac{1-s_{c}}{s_{c}}}.
\end{equation*}
 We deduce from \eqref{IF1}-\eqref{Ei} that 
\begin{equation*}
f(\|\nabla u(t)\|_{2})\leq E_{\Omega}(u(t))-l_{\Omega}(u(t))\leq E_{\Omega}(u_{0})-l<(1-\delta_{1})f(x_{1}).
\end{equation*}
Therefore, by the continuity of $\|\nabla u(t)\|_{2}$ in $t$ and \eqref{C4},  there exists $\delta_{2}>0$ such that  $\|\nabla u(t)\|^{2}_{2}\geq x^{2}_{1}+\delta_{2}$ for any $t\geq 0$. Thus, using the relation \eqref{IF1} and multiplying the viral identity \eqref{VFM} by $M[u_{0}]^{\frac{s_{c}}{1-s_{c}}}$ we obtain for $t>0$,
\begin{align}\nonumber
M(u_{0})^{\frac{1-s_{c}}{s_{c}}}J^{\prime\prime}(t)&=N(p-1)\left(E_{\Omega}(u_{0})-l_{\Omega}(u(t))\right)M[u_{0}]^{\frac{1-s_{c}}{s_{c}}}\\\nonumber
&-\left(\frac{N(p-1)-4}{2}\right)\|\nabla u(t)\|^{2}_{2}M[u_{0}]^{\frac{1-s_{c}}{s_{c}}}\\\nonumber
&-\left(\frac{N(p-1)+4}{2}\right)\int_{\mathbb{R}^{N}}V(x)|u(x)|^{2}dx\,M[u_{0}]^{\frac{1-s_{c}}{s_{c}}}\\\nonumber
&<N(p-1)\left(E_{\Omega}(u_{0})-l\right)M[u_{0}]^{\frac{1-s_{c}}{s_{c}}}\\\nonumber
&-\left(\frac{N(p-1)-4}{2}\right)\|\nabla u(t)\|^{2}_{2}M[u_{0}]^{\frac{1-s_{c}}{s_{c}}}\\\nonumber
&<N(p-1)\left(\frac{s_{c}}{N}\right)x^{2}_{1}M[u_{0}]^{\frac{1-s_{c}}{s_{c}}}-\left(\frac{N(p-1)-4}{2}\right)x^{2}_{1}M[u_{0}]^{\frac{1-s_{c}}{s_{c}}}\\ \nonumber
&-\left(\frac{N(p-1)-4}{2}\right)\delta_{2}M[u_{0}]^{\frac{1-s_{c}}{s_{c}}}\\ \label{IT}
&=-\left(\frac{N(p-1)-4}{2}\right)\delta_{2}M[u_{0}]^{\frac{1-s_{c}}{s_{c}}}.
\end{align}
Since $p>1+\frac{4}{N}$, integrating \eqref{IT} twice and taking $t$ large, the right-hand side of \eqref{IT} becomes negative, which is a contradiction. Thus, the maximum existence time is finite.

Next, we prove (iii) of lemma. Since $E_{\Omega}(u_{0})-l<0$, by \eqref{VFM} we infer that the corresponding solution blows up in finite time. Now we will show that
\begin{equation*}
\|\nabla u(t)\|_{2}\geq \left(\frac{(p-1)N}{4}\right)\left(\frac{\| Q\|_{2}}{\| u_{0}\|_{2}}\right)^{\frac{1-s_{c}}{s_{c}}}\|\nabla Q\|_{2}.
\end{equation*}
for every $t$ in the existence time. Indeed, using \eqref{GI} and multiplying both sides of  $E_{\Omega}(u)$ by $M(u)^{\frac{1-s_{c}}{s_{c}}}$ we infer that
\begin{align*}
\left(E_{\Omega}(u)-l_{\Omega}(u)\right)M(u)^{\frac{1-s_{c}}{s_{c}}}=&\frac{1}{2}\left(\|\nabla u\|_{2}\| u\|^{\frac{1-s_{c}}{s_{c}}}_{2}\right)^{2}-\frac{2}{p+1}\| u\|^{\frac{2(1-s_{c})}{s_{c}}}_{2}\| u\|^{p+1}_{2}\\
\geq &\frac{1}{2}\left(\|\nabla u\|_{2}\| u\|^{\frac{1-s_{c}}{s_{c}}}_{2}\right)^{2}-\frac{2\,c_{GN}}{p+1}\left(\|\nabla u\|_{2}\| u\|^{\frac{1-s_{c}}{s_{c}}}_{2}\right)^{\frac{n(p-1)}{2}}\\
=& h\left(\|\nabla u\|_{2}\| u\|^{\frac{1-s_{c}}{s_{c}}}_{2}\right),
\end{align*}
where 
\begin{equation*}
h(x)=\frac{1}{2}x^{2}-\frac{2\,c_{GN}}{p+1}x^{\frac{N(p-1)}{2}}, \quad \text{for $x\geq 0$.}
\end{equation*}
A simple computation shows that $h$ is increasing on $(0, x_{max})$ and decreasing on $(x_{max}, \infty)$, where 
\begin{equation*}
x_{max}=\left(\frac{2N(p-1)}{2(p+1)-N(p-1)}\right)^{\frac{1}{2}}{\|Q\|_{2}^{\frac{1}{s_{c}}}}=\|\nabla Q\|_{2}\| Q\|^{\frac{1-s_{c}}{s_{c}}}_{2}.
\end{equation*}
Moreover, we have
\begin{equation*}
h(x_{max})=\frac{s_{c}}{N}\left(\frac{2N(p-1)}{2(p+1)-N(p-1)}\right){\|Q\|_{2}^{\frac{2}{s_{c}}}}=E_{0,0}(Q)M(Q)^{\frac{(1-s_{c})}{s_{c}}}.
\end{equation*}
Notice that $h(x)>0$  for small enough $x>0$. It is not difficult to show that $h$ has a unique positive root, denoted by $x_{r}$,
\begin{equation*}
x_{r}=\left(\frac{(p-1)N}{4}\right)^{\frac{1}{s_{c}(p-1)}}\|\nabla Q\|_{2}\| Q\|^{\frac{1-s_{c}}{s_{c}}}_{2}.
\end{equation*}
On the other hand, we deduce from the condition \eqref{Fe1} that $h(\|\nabla u(t)\|_{2}\| u_{0}\|^{\frac{1-s_{c}}{s_{c}}}_{2})< 0$. Therefore, since $p>1+\frac{4}{N}$, this implies that
\begin{align*}
\|\nabla u(t)\|_{2}\| u_{0}\|^{\frac{1-s_{c}}{s_{c}}}_{2}&\geq \left(\frac{(p-1)N}{4}\right)^{\frac{1}{s_{c}(p-1)}}\|\nabla Q\|_{2}\| Q\|^{\frac{1-s_{c}}{s_{c}}}_{2}
\end{align*}
for every $t$ in the existence time, which completes the proof of Lemma \ref{lemaf}.
 \end{proof}
Now we give the proof of Theorem \ref{SC}.
\begin{proof}[ \bf {Proof of Theorem \ref{SC}}]
Let $u\in C([0, T_{+}),\Sigma)$ be the solution of \eqref{GP} with initial data $u_{0}$.\\
(i)  Notice that by \eqref{Emii} with $a=2$ we infer that
\[  |l_{\Omega}(u(t))| \leq \frac{1}{4}|\Omega|^{2}\|xu(t)\|^{2}_{2}+\|\nabla u(t)\|^{2}_{2},\]
for any $t$ as long the solution exists. Suppose that $l=-\infty$. Then there exists a sequence of times $\left\{t_{n}\right\}^{\infty}_{n=1}$ such that $\lim_{n\rightarrow \infty}|l_{\Omega}(u(t_{n}))|=\infty$. Thus, by the inequality above we see that
\[  \lim_{n\rightarrow \infty}\left[\frac{1}{4}|\Omega|^{2}\|xu(t_{n})\|^{2}_{2}+\|\nabla u(t_{n})\|^{2}_{2}   \right]=\infty.\]  
Assume by contradiction that there exists $C>0$ such that $\|\nabla u(t_{n})\|^{2}_{2}\leq C$ for all $n$. By conservation of energy and 
\eqref{Emii} we see that for $a>0$,
\[ \left(\frac{\gamma^{2}}{2}-\frac{1}{2a}|\Omega|^{2}\right) \|x u(t_{n})\|^{2}_{2}\leq C+ E_{\Omega}(u_{0})\quad \text{for all $n$,}\]
which is an absurd.  Therefore,  $\lim_{n\rightarrow\infty}\|\nabla u(t_{n})\|^{2}_{2}=\infty$. By the local theory, after extracting a subsequence, we have that $t_{n}\rightarrow T_{+}$.\\
Statements (ii) and (iii) are an immediate consequence of Lemma \ref{lemaf}. This completes the proof of theorem.
\end{proof}


\section{Existence and characterization of ground states}\label{S:2}

In this section we give the proof of the existence of ground states given in Proposition \ref{GNe}.  We define
\begin{align*}
\begin{split}
d(\omega)={\inf}\left\{S_{\omega}(u):\, u\in \Sigma \setminus  \left\{0 \right\},  I_{\omega}(u)=0\right\},\\
\mathcal{M}_{\omega}=\bigl\{ \varphi\in \Sigma: S_{\omega}(\varphi)=d(\omega),\quad  I_{\omega}(u)=0\bigl\}.
\end{split}
\end{align*}
By using the fact that $\lambda_{0}$ the smallest eigenvalue of the Schr\"odinger operator $R_{\Omega}=-\Delta +V(x)+2L_{\Omega}$ (see \eqref{impr}), we infer that $\mathfrak{t}[u]$ is bounded from below and $\mathfrak{t}[u]\geq -\lambda_{0}\|u\|^{2}_{2}$. Notice that $\sqrt{\mathfrak{t}[u]+\omega \|u\|^{2}_{2}}$ define a norm in the space $\Sigma$ for $\omega>\lambda_{0}$.
We have the following result.
\begin{lemma}\label{EEn}
Let $\omega>\lambda_{0}$. For $|\Omega|<\gamma$ we have the equivalence of norms 
\begin{equation*}
\sqrt{\mathfrak{t}[u]+\omega \|u\|^{2}_{2}}\cong \|u\|_{\Sigma}.
\end{equation*}
\end{lemma}
\begin{proof}
By the Young's inequality we infer that for any $a>0$
\begin{equation}\label{dx1}
|l_{\Omega}(v)|\leq \frac{a}{2}\|\nabla v\|^{2}_{2}+\frac{|\Omega|^{2}}{2a}\|x v\|^{2}_{2}.
\end{equation}
By using \eqref{dx1}, a simple calculation shows that there exists $C_{max}>0$ such that $\sqrt{\mathfrak{t}[u]+\omega \|u\|^{2}_{2}}\leq C_{max}\|u\|_{\Sigma}$.  On the other hand, suppose that $\mathfrak{t}[u_{n}]+\omega \|u_{n}\|^{2}_{2}\rightarrow 0$ as $n\rightarrow 0$.  From \eqref{impr} we infer that $\|u_{n}\|^{2}_{2}\rightarrow 0$ as $n\rightarrow\infty$. This implies that $\mathfrak{t}[u_{n}]\rightarrow 0$. Now, since $|\Omega|<\gamma$,   by \eqref{dx1} we see that there exists $C_{min}>0$ such that
\begin{equation*}
\mathfrak{t}[u_{n}]\geq C_{min}\left\{\|\nabla u_{n}\|^{2}_{2}+\|x u_{n}\|^{2}_{2}\right\}.
\end{equation*}
Therefore $\|u_{n}\|^{2}_{\Sigma}\rightarrow 0$ as $n\rightarrow\infty$, which completes the proof.
\end{proof}

\begin{lemma}\label{v11}
If $\omega> \lambda_{0}$, then the quantity $d(\omega)$ is positive.
\end{lemma}
\begin{proof} Let $u\in \Sigma$ be such that $I_{\omega}(u)=0$. Since $I_{\omega}(u)=0$, we infer that
\begin{equation*}
\|u\|^{p+1}_{p+1}\leq C\|u\|^{{p+1}}_{\Sigma}\leq C(\mathfrak{t}[u]+\omega \|u\|^{2}_{2})^{\frac{p+1}{2}}=C\left(\|u\|^{p+1}_{p+1}\right)^{\frac{p+1}{2}}.
\end{equation*}
This implies that
\begin{equation*}
\|u\|^{p+1}_{p+1}\geq \left(\frac{1}{C}\right)^{\frac{2}{p-1}}>0.
\end{equation*}
Therefore
\begin{equation*}
S_{\omega}(u)=\frac{1}{2}I_{\omega}(u)+\frac{p-1}{p+1}\|u\|^{p+1}_{p+1}=\frac{p-1}{p+1}\|u\|^{p+1}_{p+1}\geq\frac{p-1}{p+1}\left(\frac{1}{C}\right)^{\frac{2}{p-1}}>0.
\end{equation*}
Taking the infimum, we get $d(\omega)>0$. 
\end{proof}

\begin{lemma}\label{v22}
Let $\omega> \lambda_{0}$. The set $\mathcal{M}_{\omega}$ is non-empty.
\end{lemma}
\begin{proof} 
Let $\left\{u_{n}\right\}^{\infty}_{n=1}$ be  a minimizing sequence of $d(\omega)$. Since $S_{\omega}(u_{n})=\frac{1}{2}I_{\omega}(u_{n})+\frac{p-1}{p+1}\|u\|^{p+1}_{p+1}\rightarrow d(\omega)$ as $n$ goes to $\infty$, we infer that $\|u_{n}\|^{p+1}_{p+1}$ is bounded. Thus, from $I_{\omega}(u_{n})=0$
we obtain that $\mathfrak{t}[u_{n}]+\omega \|u_{n}\|^{2}_{2}$ is bounded in $\Sigma$. Therefore, there exists $u\in \Sigma$ such that, up to sequence,  $u_{n}\rightharpoonup u_{0}$ weakly in $\Sigma$ and 
\begin{equation}\label{sli}
\mathfrak{t}[u_{0}]+\omega \|u_{0}\|^{2}_{2}\leq\liminf_{n\rightarrow \infty}\left\{\mathfrak{t}[u_{n}]+\omega \|u_{n}\|^{2}_{2}\right\}.
\end{equation}
Now, since $\Sigma\hookrightarrow L^{p+1}$ is compact for $1\leq p<2^{\ast}$, we have $u_{n}\rightarrow u_{0}$ strongly in $L^{p+1}$. By \eqref{sli}, this implies 
\begin{equation*}
I_{\omega}(u_{0})\leq \liminf_{n\rightarrow \infty}\left\{  \mathfrak{t}[u_{n}]+\omega \|u_{n}\|^{2}_{2}-2\|u_{n}\|^{p+1}_{p+1}\right\}= \liminf_{n\rightarrow \infty}I_{\omega}(u_{n})=0
\end{equation*}
and we  also have
\begin{equation*}
d(\omega)=\lim_{n\rightarrow\infty}S_{\omega}(u_{n})=\lim_{n\rightarrow\infty}\frac{p-1}{p+1}\|u_{n}\|^{p+1}_{p+1}=\frac{p-1}{p+1}\|u_{0}\|^{p+1}_{p+1}.
\end{equation*}
We claim that $u_{0}\in \mathcal{M}_{\omega}$. To show this we only need to show that $I_{\omega}(u_{0})=0$. To see this, suppose that $I_{\omega}(u_{0})<0$.  For $\kappa>0$ we see that
\begin{equation*}
\kappa^{-2}I_{\omega}(\kappa u_{0})=\mathfrak{t}[u_{0}]+\omega \|u_{0}\|^{2}_{2}-2\kappa^{p-1}\|u_{0}\|^{p+1}_{p+1}.
\end{equation*}
A simple calculation shows that the only solution to the equation  $\kappa^{-2}I_{\omega}(\kappa u_{0})=0$ is
\begin{equation*}
\kappa_{0}=\left(\frac{\mathfrak{t}[u_{0}]+\omega \|u_{0}\|^{2}_{2}}{2\|u_{0}\|^{p+1}_{p+1}}\right)^{\frac{1}{p-1}}.
\end{equation*}
Notice that $0<\kappa_{0}<1$. Now, since $I_{\omega}(\kappa_{0}u_{0})=0$,  by definition of $d(\omega)$ we see that
\begin{equation*}
\frac{p-1}{p+1}\|u_{0}\|^{p+1}_{p+1}=d(\omega)\leq \frac{p-1}{p+1}\|\kappa_{0}u_{0}\|^{p+1}_{p+1}=\kappa^{p+1}_{0}\frac{p-1}{p+1}\|u_{0}\|^{p+1}_{p+1}<\frac{p-1}{p+1}\|u_{0}\|^{p+1}_{p+1},
\end{equation*}
which is a contradiction. Therefore $d(\omega)=S_{\omega}(u_{0})$ and $I_{\omega}(u_{0})=0$; that is, $u_{0}\in \mathcal{M}_{\omega}$. This completes the proof of the lemma.
\end{proof}

\begin{lemma}\label{v33}
If $\omega> \lambda_{0}$, then $\mathcal{G}_{\omega}=\mathcal{M}_{\omega}$. 
\end{lemma}
\begin{proof} 
First we show that we have $\mathcal{M}_{\omega}\subset \mathcal{G}_{\omega}$.  Let $\varphi\in \mathcal{M}_{\omega}$. Then there exists a Lagrange multiplier $\kappa\in \mathbb{R}$ such that $S^{\prime}_{\omega}(\varphi)=\kappa I^{\prime}_{\omega}(\varphi)$. A simple calculation shows that
\begin{equation*}
0=I_{\omega}(\varphi)=\left\langle S^{\prime}_{\omega}(\varphi),\varphi\right\rangle=\kappa \left\langle I^{\prime}_{\omega}(\varphi),\varphi\right\rangle.
\end{equation*}
Moreover, since $\mathfrak{t}[\varphi]+\omega \|\varphi\|^{2}_{2}=2\|\varphi\|^{p+1}_{p+1}$, we get
\begin{equation*}
\left\langle I^{\prime}_{\omega}(\varphi),\varphi\right\rangle=2\mathfrak{t}[\varphi]+2\omega \|\varphi\|^{2}_{2}-2(p+1)\|\varphi\|^{p+1}_{p+1}
=-2(p-1)\|\varphi\|^{p+1}_{p+1}<0.
\end{equation*}
Therefore, $\kappa=0$. This implies that $\varphi$ satisfies the stationary problem \eqref{EE}; that is, $\varphi\in \mathcal{A}_{\omega}$.
Next if $v\in  \mathcal{A}_{\omega}$, then $I_{\omega}(v)=\left\langle S^{\prime}_{\omega}(v),v\right\rangle=0$ and, since $\varphi\in \mathcal{M}_{\omega}$, we infer that $S_{\omega}(\varphi)\leq S_{\omega}(v)$; that is, 
\begin{equation*}
S_{\omega}(\varphi)\leq S_{\omega}(v)\quad \text{for all $v\in  \mathcal{A}_{\omega}$.}
\end{equation*}
Hence $\varphi\in \mathcal{G}_{\omega}$ and $\mathcal{M}_{\omega}\subset \mathcal{G}_{\omega}$.  In particular, $\mathcal{G}_{\omega}$ is not-empty. On the other hand,  let $\varphi\in \mathcal{G}_{\omega}$. Since $\mathcal{G}_{\omega}\subset \mathcal{A}_{\omega}$ we see that 
$I_{\omega}(\varphi)=0$. Moreover, by using the fact that $\mathcal{M}_{\omega}\subset \mathcal{G}_{\omega}$, we infer that $d(\omega)=S_{\omega}(\varphi)$. Thus, $\varphi\in \mathcal{M}_{\omega}$, which completes the proof.
\end{proof}

\begin{proof}[ \bf {Proof of Proposition \ref{GNe}}]
The proof of Proposition \ref{GNe} is a consequence of the Lemmas \ref{v11}, \ref{v22} and \ref{v33}.
\end{proof}



\section{Instability of standing waves}\label{S:4}
Concerning the sufficient condition for instability, i.e., Theorem \ref{Ax1}, the proof  follows from exactly the same argument in \cite[Proposition 1.1]{FOIB2003} and we omit the details.  In order to prove Collorary \ref{Eis} we establish some notation and a lemma. We follow closely the approach of Fukuizumi and Ohta \cite{FOIB2003}.
Let $\phi_{\omega}\in  \mathcal{G}_{\omega}$. We set the rescaled function
\begin{equation}\label{REf}
\phi_{\omega}(x)=\omega^{\frac{1}{p-1}}\widetilde{\phi}_{\omega}(\sqrt{\omega}x)\quad \text{for $\omega>0$.}
\end{equation}
Since $[L_{\Omega}\phi_{\omega}](x)=\omega^{\frac{1}{p-1}}[L_{\Omega}\widetilde{\phi}_{\omega}](\sqrt{\omega}x)$
and $V(x)=\omega^{-1}V(\sqrt{\omega}x)$, by \eqref{EE}, it is not difficult to prove that $\widetilde{\phi}_{\omega}(x)$ satisfies the elliptic equation
\begin{equation}\label{Nie}
 -\Delta \varphi+\varphi+\omega^{-2}V(x)\varphi-2|\varphi|^{p-1}\varphi+2\omega^{-1}L_{\Omega}\varphi=0, \quad x\in\mathbb{R}^{N}.
\end{equation}
Moreover, as $l_{\Omega}(\phi_{\omega})=\omega^{\frac{2}{p-1}}\omega^{-\frac{N}{2}}l_{\Omega}(\widetilde{\phi}_{\omega})$,
we have
\begin{equation}\label{Eps1}
\frac{\int_{\mathbb{R}^{N}}V(x)|{\phi}_{\omega}(x)|^{2}\,dx+2l_{\Omega}({\phi}_{\omega})}{{ \|{\phi}_{\omega}\|^{p+1}_{p+1}}}=\frac{\omega^{-2}\int_{\mathbb{R}^{N}}V(x)|\widetilde{\phi}_{\omega}(x)|^{2}\,dx+2\omega^{-1}l_{\Omega}(\widetilde{\phi}_{\omega})}{ \|\widetilde{\phi}_{\omega}\|^{p+1}_{p+1}}.
\end{equation}
\begin{lemma}\label{Ls1} 
Let $\gamma>0$, $\eta>0$ and $\widetilde{\phi}_{\omega}(x)$ be the rescaled function given in \eqref{REf}. There is a $\epsilon>0$ small enough (depending only on $\eta$, $N$ and $p$) such that if ${|\Omega|^{2}}\leq\epsilon {\gamma^{2}}$, then there exists a sequence $\left\{\omega_{n}\right\}^{\infty}_{n=1}$ that satisfies
\begin{equation*}
\lim_{n\rightarrow\infty}[\omega_{n}^{-2}\int_{\mathbb{R}^{N}}V(x)|\widetilde{\phi}_{\omega_{n}}(x)|^{2}\,dx+2\omega_{n}^{-1}l_{\Omega}(\widetilde{\phi}_{\omega_{n}})]\leq
4\eta.
\end{equation*}
Moreover, $\omega_{n}\rightarrow \infty$ as $n\rightarrow \infty$. 
\end{lemma}
\begin{proof}
Let $Q$ be the unique positive ground state for \eqref{Izx}. It is well-known that 
\begin{equation}\label{CVQ1}
\|Q\|^{p+1}_{p+1}=\inf\bigl\{ \|v\|^{p+1}_{p+1}: v\in \Sigma\setminus  \left\{0 \right\}: {I}_{0,1}(v)\leq 0\bigl\},
\end{equation}
where
\begin{equation*}
{I}_{0,1}(v)=\int_{\mathbb{R}^{N}}|\nabla v|^{2}dx +\int_{\mathbb{R}^{N}}|v|^{2}dx-{2}\int_{\mathbb{R}^{N}}|v|^{p+1}dx.
\end{equation*}
By Proposition \ref{GNe}, it is not difficult to show that (see proof of Lemma 3.1 in \cite{FOIB2003})
\begin{equation*}
\|{\phi}_{\omega}\|^{p+1}_{p+1}=\inf\bigl\{ \|v\|^{p+1}_{p+1}: v\in \Sigma\setminus  \left\{0 \right\}:{I}_{\omega}(v)\leq 0\bigl\}.
\end{equation*}
Thus, by using \eqref{REf} we infer that
\begin{equation}\label{Vci}
\|\widetilde{\phi}_{\omega}\|^{p+1}_{p+1}=\inf\bigl\{ \|v\|^{p+1}_{p+1}: v\in \Sigma\setminus  \left\{0 \right\}: \tilde{I}_{\omega}(v)\leq 0\bigl\},
\end{equation}
where
\[\begin{split}
\tilde{I}_{\omega}(v)=\int_{\mathbb{R}^{N}}|\nabla v|^{2}dx+\omega^{-2}\int_{\mathbb{R}^{N}}V(x)|v(x)|^{2}\,dx  +\int_{\mathbb{R}^{N}}|v|^{2}dx\\
-{2}\int_{\mathbb{R}^{N}}|v|^{p+1}dx+2\omega^{-1}l_{\Omega}(v).
\end{split}\]

Let $\eta>0$ small. We assume that  ${|\Omega|^{2}}\leq\epsilon {\gamma^{2}}$ for some $\epsilon>0$ small enough to be choose later (depending only on $\eta$, $N$ and $p$).

We claim that there exist $\epsilon$ small enough (depending only on $\eta$, $N$ and $p$) and $\omega^{\ast}$
large enough (depending only on $\eta$) such that
\begin{equation}\label{claim1}
\|Q\|^{p+1}_{p+1}-\eta \leq\|\widetilde{\phi}_{\omega}\|^{p+1}_{p+1}\leq \|Q\|^{p+1}_{p+1}+\eta \quad \mbox{when $\omega\in (\omega^{\ast}, \infty)$}.
\end{equation}
We show this claim in two steps.\\
\textit{Step 1.} Let  $\kappa>1$ close to $1$. We show that there exists $\omega(\kappa)>0$ such that
\begin{equation}\label{laa1}
\|\widetilde{\phi}_{\omega}\|^{p+1}_{p+1}\leq \kappa^{p+1}\|Q\|^{p+1}_{p+1}\quad \text{for every $\omega\in (\omega(\kappa), \infty)$}.
\end{equation}
First, note that there exists $\omega(\kappa)>0$ such that  $\tilde{I}_{\omega}(\kappa Q)<0$ holds for all $\omega\in (\omega(\kappa), \infty)$. Indeed, since $Q$ is radial and ${I}_{0,1}(Q)=0$, it follows that
\begin{equation}\label{Cp}
\begin{split}
\kappa^{-2}\tilde{I}_{\omega}(\kappa Q)=-2(\kappa^{p-1}-1)\|Q\|^{p+1}_{p+1}+\omega^{-2}\int_{\mathbb{R}^{N}}V(x)|Q(x)|^{2}\,dx.
\end{split}
\end{equation}
Moreover, it is well known that $Q$ has an exponential decay at infinity, thus
\begin{equation}\label{Ed}
\lim_{\omega\rightarrow \infty}\omega^{-2}\int_{\mathbb{R}^{N}}V(x)|Q(x)|^{2}\,dx=0.
\end{equation}
Therefore, from \eqref{Cp} and \eqref{Ed}, we infer that  there exists $\omega(\kappa)>0$ such that  $\tilde{I}_{\omega}(\kappa Q)<0$ holds for all $\omega\in (\omega(\kappa), \infty)$. By the variational characterization \eqref{Vci} of $\widetilde{\phi}_{\omega}$,  we see that 
$\|\widetilde{\phi}_{\omega}\|^{p+1}_{p+1}\leq \kappa^{p+1}\|Q\|^{p+1}_{p+1}$  {when $\omega\in (\omega(\kappa), \infty)$}. In particular, for $\kappa=(1+\eta/\|Q\|^{p+1}_{p+1})^{1/(p+1)}$, we infer that 
\begin{equation}\label{DG1112}
\|\widetilde{\phi}_{\omega}\|^{p+1}_{p+1}\leq \|Q\|^{p+1}_{p+1}+\eta \quad \mbox{when $\omega\in (\omega(\eta), \infty)$}.
\end{equation}
\textit{Step 2.}
First, notice that
\[{I}_{0,1}(\kappa \widetilde{\phi}_{\omega})=\kappa^{2}[\tilde{I}_{\omega}( \widetilde{\phi}_{\omega})
-2(\kappa^{p-1}-1)-\omega^{-2}\int_{\mathbb{R}^{N}}V(x)|\widetilde{\phi}_{\omega}(x)|^{2}\,dx-2\omega^{-1}l_{\Omega}(\widetilde{\phi}_{\omega})].
\]
Therefore, since $\tilde{I}_{\omega}(\widetilde{\phi}_{\omega})=0$, it follows that
\[\begin{split}
\kappa^{-2}{I}_{0,1}(\kappa \widetilde{\phi}_{\omega})=-2(\kappa^{p-1}-1)\|\widetilde{\phi}_{\omega}\|^{p+1}_{p+1}-\omega^{-2}\int_{\mathbb{R}^{N}}V(x)|\widetilde{\phi}_{\omega}(x)|^{2}\,dx-2\omega^{-1}l_{\Omega}(\widetilde{\phi}_{\omega}).
\end{split}\]
Thus, from inequality \eqref{SnI} (see Lemma \ref{Lesim} below) we have
\begin{equation}\label{Er5}
\begin{split}
\kappa^{-2}{I}_{0,1}(\kappa \widetilde{\phi}_{\omega})\leq f\left(\frac{|\Omega|^{2}}{\gamma^{2}}\right)\|\widetilde{\phi}_{\omega}\|^{p+1}_{p+1},
\quad \mbox{for every $\omega>0$},
\end{split}
\end{equation}
where
\[
\begin{split}
f(x):=-2(\kappa^{p-1}-1)+x\left(\frac{N(p-1)}{p+1}\right)+\frac{2x}{1-x}.
\end{split}
\]
We recall that   ${|\Omega|^{2}}\leq\epsilon {\gamma^{2}}$. Thus, for $\kappa>1$, there exists $\epsilon>0$ (depending on $\kappa$, $N$ and $p$) small enough such that
$f\left(\frac{|\Omega|^{2}}{\gamma^{2}}\right)\leq f(\epsilon)<0$. From \eqref{Er5}, it follows that
\[{I}_{0,1}(\kappa \widetilde{\phi}_{\omega})<0 \quad \mbox{when $\omega>0$},\]
and, by the characterization variational \eqref{CVQ1} we obtain $\|Q\|^{p+1}_{p+1}\leq \kappa^{p+1}\|\widetilde{\phi}_{\omega}\|^{p+1}_{p+1}$
for every $\omega>0$. In particular, for $\kappa=(1+\eta/2\|Q\|^{p+1}_{p+1})^{1/(p+1)}$, we infer that there exists $\epsilon>0$ (depending only on $\eta$, $N$ and $p$)
such that 
\begin{equation}\label{DG2}
\|Q\|^{p+1}_{p+1}\leq \|\widetilde{\phi}_{\omega}\|^{p+1}_{p+1}+\eta \quad \mbox{when $\omega\in (\omega(\eta), \infty)$}.
\end{equation}
Here we have used that  $\|\widetilde{\phi}_{\omega}\|^{p+1}_{p+1}/ 2\|Q\|^{p+1}_{p+1}\leq 1$ {for every $\omega\in (\omega(\eta), \infty)$} (see \eqref{laa1}). Then the claim follows from \eqref{DG1112} and \eqref{DG2}. 

On the other hand, we set 
\begin{equation}\label{bdefij}
b^{p-1}_{\omega}:=1+ \frac{{I}_{0,1}(\widetilde{\phi}_{\omega})}{2\|\widetilde{\phi}_{\omega}\|^{p+1}_{{p+1}}}.
\end{equation}
It is not difficult to show that  ${I}_{0,1}(b_{\omega}\widetilde{\phi}_{\omega})=0$. We claim that there exists a subsequence $\left\{\widetilde{\phi}_{\omega_{n}} \right\}$, with $\omega_{n}\rightarrow\infty$ as $n\rightarrow\infty$, such that either $b_{\omega_{n}}\geq1$ for  $n\geq 1$ or $b_{\omega_{n}}\rightarrow1$ as $n\rightarrow\infty$.
Indeed, we have two possibilities:\\
(i) There exists a subsequence $\omega_{n}$ such that ${I}_{0,1}(\widetilde{\phi}_{\omega_{n}})>0$ for all $n\geq 1$. In this case,
$b_{\omega_{n}}\geq1$ for $n\geq1$.\\
(ii) If (i) is false, then there exists $\omega^{\ast}>0$ such that ${I}_{0,1}(\widetilde{\phi}_{\omega})\leq 0$  for every  $\omega>\omega^{\ast}$. In this case, by the variational characterization \eqref{CVQ1} we obtain that
\begin{equation}\label{Cvc1122}
\|Q\|^{p+1}_{p+1}\leq \|\widetilde{\phi}_{\omega}\|^{p+1}_{p+1}, \quad \text{for all $\omega>\omega^{\ast}$}.
\end{equation}
Moreover, from \textit{Step 1} above (in, particular by \eqref{laa1}) we see that 
\begin{equation}\label{limitun}
\lim_{\omega\rightarrow \infty} \|\widetilde{\phi}_{\omega}\|^{p+1}_{p+1}\leq \|Q\|^{p+1}_{p+1}.
\end{equation}
Combining \eqref{Cvc1122} and \eqref{limitun} we infer that
\[
 \|Q\|^{p+1}_{p+1}\leq \lim_{\omega\rightarrow \infty} \|\widetilde{\phi}_{\omega}\|^{p+1}_{p+1}
\leq  \|Q\|^{p+1}_{p+1},
\]
that is,
\begin{equation}\label{igualdad}
\lim_{\omega\rightarrow \infty} \|\widetilde{\phi}_{\omega}\|^{p+1}_{p+1}= \|Q\|^{p+1}_{p+1}.
\end{equation}
Now since ${I}_{0,1}(b_{\omega}\widetilde{\phi}_{\omega})=0$, it follows that by \eqref{CVQ1},
\[
\|Q\|^{p+1}_{p+1}\leq  \|b_{\omega} \widetilde{\phi}_{\omega}\|^{p+1}_{p+1}\leq b^{p+1}_{\omega}\|\widetilde{\phi}_{\omega}\|^{p+1}_{p+1}.
\]
Thus, by \eqref{igualdad} we have
\[
\liminf_{\omega\rightarrow \infty}b_{\omega}\geq \liminf_{\omega\rightarrow \infty}\frac{\|Q\|_{p+1}}{\|\widetilde{\phi}_{\omega}\|_{p+1}}=1.
\]
This implies by \eqref{bdefij},
\[
\liminf_{\omega\rightarrow \infty} {I}_{0,1}(\widetilde{\phi}_{\omega})=
\liminf_{\omega\rightarrow \infty} 2(b^{p-1}_{\omega}-1)
 \|\widetilde{\phi}_{\omega}\|^{p+1}_{p+1}\geq 0.
\]
Finally, as ${I}_{0,1}(\widetilde{\phi}_{\omega})\leq 0$  for all  $\omega>\omega^{\ast}$, the inequality above shows that there exists a subsequence $\omega_{n}$ such that
\[\lim_{n\rightarrow \infty}{I}_{0,1}(\widetilde{\phi}_{\omega_{n}})=0.\] 
Therefore, by using \eqref{bdefij} and \eqref{igualdad}, we get
\[
\lim_{n\rightarrow \infty}b_{\omega_{n}}=1.
\]
In any case, there exists a subsequence $\left\{\widetilde{\phi}_{\omega_{n}} \right\}$, with $\omega_{n}\rightarrow\infty$ as $n\rightarrow\infty$, such that either $b_{\omega_{n}}\geq1$ for $n\geq 1$ or $\lim_{n\rightarrow \infty}b_{\omega_{n}}=1$.
This proves the claim.

 Next, since  ${I}_{0,1}(Q)=0$, ${I}_{0,1}(b_{\omega_{n}}\widetilde{\phi}_{\omega_{n}})=0$ and \eqref{claim1}, we obtain
\[\begin{split}
\lim_{n\rightarrow\infty}[\|\nabla \widetilde{\phi}_{\omega_{n}}\|^{2}_{2}+ \|\widetilde{\phi}_{\omega_{n}}\|^{2}_{2}]
=2\lim_{n\rightarrow\infty}b_{\omega_{n}}^{p-1}\|\widetilde{\phi}_{\omega_{n}}\|^{p+1}_{p+1}\geq 2 \lim_{n\rightarrow\infty}\|\widetilde{\phi}_{\omega_{n}}\|^{p+1}_{p+1}\\
\geq  2[\|Q\|^{p+1}_{p+1}-\eta] 
=\|\nabla Q\|^{2}_{2}+ \|Q\|^{2}_{2}-2\eta.
\end{split}\]
Finally, since $\tilde{I}_{\omega}(\widetilde{\phi}_{\omega_{n}})=0$, by inequality above and \eqref{claim1} we have
\[\begin{split}
\lim_{n\rightarrow\infty}[\omega_{n}^{-2}\int_{\mathbb{R}^{N}}V(x)|\widetilde{\phi}_{\omega_{n}}(x)|^{2}\,dx+2\omega_{n}^{-1}l_{\Omega}(\widetilde{\phi}_{\omega_{n}})]
&\leq 2\|Q\|^{p+1}_{p+1}-\|\nabla Q\|^{2}_{2}- \|Q\|^{2}_{2}+4\eta\\
&= 4\eta.
\end{split}\]
 This completes the proof.
\end{proof}

\begin{lemma}\label{Lesim}
Let $\phi_{\omega}\in  \mathcal{G}_{\omega}$ and consider the rescaled function $\phi_{\omega}(x)=\omega^{\frac{1}{p-1}}\widetilde{\phi}_{\omega}(\sqrt{\omega}x)$. The following fact hold.\\
(i) We have the following Pohozaev identity
\begin{equation}\label{Ew1}
 \|\nabla \widetilde{\phi}_{\omega}\|^{2}_{2}-\omega^{-2}\int_{\mathbb{R}^{N}}V(x)|\widetilde{\phi}_{\omega}(x)|^{2}\,dx-\frac{N(p-1)}{p+1}\|\widetilde{\phi}_{\omega}\|^{p+1}_{p+1}=0.
\end{equation}
(ii) For all $\omega>0$,
\begin{equation}\label{SnI}
\begin{split}
 -\omega^{-2}\int_{\mathbb{R}^{N}}V(x)|\widetilde{\phi}_{\omega}(x)|^{2}\,dx-2\omega^{-1}l_{\Omega}(\widetilde{\phi}_{\omega})\leq\\
\frac{|\Omega|^{2}}{\gamma^{2}}
\left[ \frac{N(p-1)}{p+1}+\frac{2\gamma^{2}}{\gamma^{2}-|\Omega|^{2}}\right] \|\widetilde{\phi}_{\omega}\|^{p+1}_{L^{p+1}}.
\end{split}
\end{equation}
\end{lemma}
\begin{proof} 
 We set
\[\begin{split}
\widetilde{S}(\varphi):=\frac{1}{2}\|\nabla \varphi\|^{2}_{2}+\frac{1}{2}\| \varphi\|^{2}_{2}
+\frac{\omega^{-2}}{2}\int_{\mathbb{R}^{N}}V(x)|\varphi(x)|^{2}dx
-\frac{2}{p+1}\| \varphi\|^{p+1}_{{p+1}}+\omega^{-1}l_{\Omega}(\varphi).
\end{split}\]
Note that $\widetilde{S}^{\prime}(\varphi)=\mbox{LHS\eqref{Nie}}$. Since $\widetilde{\phi}_{\omega}$ is solution of \eqref{Nie} we see that
$\widetilde{S}^{\prime}(\widetilde{\phi}_{\omega})=0$ and
\begin{equation}\label{Tp1}
\left.\frac{\partial}{\partial \lambda} \widetilde{S}^{\prime}(\widetilde{\phi}^{\lambda}_{\omega})\right|_{\lambda=1}  
=\left\langle  \widetilde{S}^{\prime}(\widetilde{\phi}_{\omega}), \left.\frac{\partial}{\partial \lambda} \widetilde{\phi}^{\lambda}_{\omega}\right|_{\lambda=1}\right\rangle=0,
\end{equation}
where $\widetilde{\phi}^{\lambda}_{\omega}(x)=\lambda^{\frac{N}{2}}\widetilde{\phi}_{\omega}(\lambda x)$. By using the fact that
\[\begin{split}
\widetilde{S}(\widetilde{\phi}^{\lambda}_{\omega})=
\frac{\lambda^{2}}{2}\|\nabla \widetilde{\phi}_{\omega}\|^{2}_{{2}}+\frac{1}{2}\|\widetilde{\phi}_{\omega}\|^{2}_{{2}}
+\frac{\omega^{-2}\lambda^{-2}}{2}\int_{\mathbb{R}^{N}}V(x)|\widetilde{\phi}_{\omega}(x)|^{2}dx\\
-\frac{2\lambda^{\frac{N(p-1)}{2}}}{p+1}\| \widetilde{\phi}_{\omega}\|^{p+1}_{{p+1}}+\omega^{-1}l_{\Omega}(\widetilde{\phi}_{\omega})
\end{split}\]
we obtain 
\begin{equation}\label{Tph23}
 \left.\frac{\partial}{\partial \lambda} \widetilde{S}^{\prime}(\widetilde{\phi}^{\lambda}_{\omega})\right|_{\lambda=1}
=\mbox{LHS}\eqref{Ew1}.
\end{equation}
Combining \eqref{Tp1} and \eqref{Tph23} proves \eqref{Ew1}. This proves statement (i) in the lemma. To prove statement (ii) we first note that
\begin{equation}
 \label{Ew2}
 \omega^{-2}\int_{\mathbb{R}^{N}}V(x)|\widetilde{\phi}_{\omega}(x)|^{2}\,dx\leq \frac{2\gamma^{2}}{\gamma^{2}-|\Omega|^{2}}\| \widetilde{\phi}_{\omega}\|^{p+1}_{p+1}.
\end{equation}
Indeed, by \eqref{Emii} we obtain
\[\begin{split}
|l_{\Omega}(\widetilde{\phi}_{\omega})| &\leq \frac{\omega}{2}
|\Omega|^{2}\|x\widetilde{\phi}_{\omega}\|^{2}_{2}
+\frac{1}{2\omega}\|\nabla \widetilde{\phi}_{\omega}\|^{2}_{2}\\
&\leq \frac{\omega}{2}\frac{|\Omega|^{2}}{\gamma^{2}}
\int_{\mathbb{R}^{N}}V(x)|\widetilde{\phi}_{\omega}(x)|^{2}\,dx
+\frac{1}{2\omega}\|\nabla \widetilde{\phi}_{\omega}\|^{2}_{2}
\end{split}\]
Thus, since $\tilde{I}_{\omega}(\widetilde{\phi}_{\omega})=0$, by inequality above we infer that
\begin{align*}
\omega^{-2}\int_{\mathbb{R}^{N}}V(x)|\widetilde{\phi}_{\omega}(x)|^{2}dx&\leq 2\| \widetilde{\phi}_{\omega}\|^{p+1}_{p+1} -\|\nabla \widetilde{\phi}_{\omega}\|^{2}_{L^{2}}+2\omega^{-1}|l_{\Omega}(\widetilde{\phi}_{\omega})|\\
&\leq 2\| \widetilde{\phi}_{\omega}\|^{p+1}_{p+1}+\omega^{-2}\frac{|\Omega|^{2}}{\gamma^{2}}\int_{\mathbb{R}^{N}}V(x)|\widetilde{\phi}_{\omega}(x)|^{2}dx.
\end{align*} 
This inequality implies \eqref{Ew2}. On the other hand, notice that
\begin{equation}\label{Er1}
 -\omega^{-2}\int_{\mathbb{R}^{N}}V(x)|\widetilde{\phi}_{\omega}(x)|^{2}\,dx-2\omega^{-1}l_{\Omega}(\widetilde{\phi}_{\omega})\leq\frac{|\Omega|^{2}}{\gamma^{2}} \|\nabla \widetilde{\phi}_{\omega}\|^{2}_{2}.
\end{equation}
Indeed, the inequality is trivial if $l_{\Omega}(\widetilde{\phi}_{\omega})\geq 0$. Now if $l_{\Omega}(\widetilde{\phi}_{\omega})< 0$, then 
$-2\omega^{-1}l_{\Omega}(\widetilde{\phi}_{\omega})=2\omega^{-1}|l_{\Omega}(\widetilde{\phi}_{\omega})|$. Thus, the inequality \eqref{Er1} follows from inequality \eqref{Emii} with $a=\omega|\Omega|^{2}/\gamma^{2}$.

Finally, combining \eqref{Er1}, \eqref{Ew1} and \eqref{Ew2}  we get  \eqref{SnI}.
\end{proof}

\begin{proof}[ \bf {Proof of Corollary \ref{Eis}}] Set $\phi^{s}_{\omega}(x):=s^{\frac{N}{2}}\phi_{\omega}(sx)$. Some straightforward computations revel that
\begin{equation*}
E_{\Omega}(\phi^{s}_{\omega})=\frac{s^{2}}{2}\|\nabla \phi_{\omega}\|^{2}_{2}
+\frac{1}{2}\int_{ \mathbb{R}^{N}}\frac{V(x)}{s^{2}}|\phi_{\omega}(x)|^{2}\,dx-\frac{2s^{\frac{N(p-1)}{2}}}{p+1}\|\phi_{\omega}\|^{p+1}_{p+1}+l_{\Omega}(\phi_{\omega}).
\end{equation*}
Since $P(\phi_{\omega})=\partial_{s}S_{\omega}(\phi^{s}_{\omega})|_{s=1}=0$, it follows that
\[
\partial^{2}_{s}E_{\Omega}(\phi^{s}_{\omega})|_{s=1}=4\int_{ \mathbb{R}^{N}}{V(x)}|\phi_{\omega}(x)|^{2}\,dx- N\frac{p-1}{p+1}\left(\frac{N(p-1)}{2}-2\right)\|\phi_{\omega} \|^{p+1}_{p+1}.
\]
Thus, that the condition $\partial^{2}_{s}E_{\Omega}(\phi^{s}_{\omega_{n}})|_{s=1}<0$  is equivalent to 
\begin{equation}\label{Aesd}
4\frac{\int_{ \mathbb{R}^{N}}{V(x)}|\phi_{\omega}(x)|^{2}\,dx}{\|\phi_{\omega} \|^{p+1}_{p+1}}< N\frac{p-1}{p+1}\left(\frac{N(p-1)}{2}-2\right).
\end{equation}
We observe that since $p>1+\frac{4}{N}$, this implies that $\mbox{RHS\eqref{Aesd}}>0$. 

On the other hand, since $P(\phi_{\omega})=0$, we obtain
\begin{equation}\label{PHOq}
 \|\nabla{\phi}_{\omega}\|^{2}_{2}-\int_{\mathbb{R}^{N}}V(x)|{\phi}_{\omega}(x)|^{2}\,dx-\frac{N(p-1)}{p+1}\|{\phi}_{\omega}\|^{p+1}_{p+1}=0.
\end{equation}
Moreover, it is not difficult to show (see proof of \eqref{Ew2} in Lemma \ref{Lesim}) that
\begin{equation} \label{sdp}
\int_{\mathbb{R}^{N}}V(x)|{\phi}_{\omega}(x)|^{2}\,dx\leq \frac{2\gamma^{2}}{\gamma^{2}-|\Omega|^{2}}\|{\phi}_{\omega}\|^{p+1}_{p+1}.
\end{equation}
Combining \eqref{sdp} and the identity  \eqref{PHOq} we obtain $\|\nabla \phi_{\omega}\|^{2}_{2}\leq \beta\|\phi_{\omega} \|^{p+1}_{p+1}$, where
\[\beta:=  \frac{2\gamma^{2}}{\gamma^{2}-|\Omega|^{2}}+\frac{N(p-1)}{p+1}.\]
From \eqref{Emii} with $a=2|\Omega|^{2}/\gamma^{2}$ we have
\[
2|l_{\Omega}( \phi_{\omega})|\leq \frac{2|\Omega|^{2}}{\gamma^{2}}\|\nabla \phi_{\omega}\|^{2}_{2}+\frac{1}{2}\int_{\mathbb{R}^{N}}V(x)|{\phi}_{\omega}(x)|^{2}dx.
\]
This implies that
\begin{equation}\label{Fem}
\frac{2|l_{\Omega}( \phi_{\omega})|}{\|\phi_{\omega} \|^{p+1}_{p+1}}\leq \frac{2\beta|\Omega|^{2}}{\gamma^{2}}+\frac{1}{2}\frac{\int_{\mathbb{R}^{N}}V(x)|{\phi}_{\omega}(x)|^{2}dx}{\|\phi_{\omega} \|^{p+1}_{p+1}}.
\end{equation}
Now by Lemma \ref{Ls1},  given $\eta>0$ there exists a sequence $\left\{\omega_{n}\right\}$ such that 
\begin{equation}\label{Lhd}
\omega_{n}^{-2}\int_{\mathbb{R}^{N}}V(x)|\widetilde{\phi}_{\omega_{n}}(x)|^{2}\,dx+2\omega_{n}^{-1}l_{\Omega}(\widetilde{\phi}_{\omega_{n}})\leq
4\eta\quad \mbox{for sufficiently large  $n$,}
\end{equation}
where $\omega_{n}\rightarrow \infty$ as $n\rightarrow \infty$. Moreover
\begin{equation}\label{hc1}
\|Q\|^{p+1}_{p+1}-\eta \leq\|\widetilde{\phi}_{\omega_{n}}\|^{p+1}_{p+1}\leq \|Q\|^{p+1}_{p+1}+\eta, \quad \mbox{for sufficiently large  $n$.}
\end{equation}
From \eqref{Eps1}, \eqref{hc1} and \eqref{Lhd} we get
\begin{equation}\label{Cvb1}
\begin{split}
\frac{\int_{\mathbb{R}^{N}}V(x)|{\phi}_{\omega_{n}}(x)|^{2}dx}{\|\phi_{\omega_{n}} \|^{p+1}_{p+1}}\leq
\frac{4\eta}{(\|Q\|^{p+1}_{L^{p+1}}-\eta)}+2\frac{|l_{\Omega}( \phi_{\omega_{n}})|}{\|\phi_{\omega_{n}} \|^{p+1}_{p+1}}.
\end{split}
\end{equation}
Combining \eqref{Fem} and \eqref{Cvb1}, it follows from straightforward calculation that
\[\begin{split}
\frac{1}{2}\frac{\int_{\mathbb{R}^{N}}V(x)|{\phi}_{\omega_{n}}(x)|^{2}dx}{\|\phi_{\omega_{n}} \|^{p+1}_{p+1}}\leq
\frac{4\eta}{(\|Q\|^{p+1}_{L^{p+1}}-\eta)}+2\epsilon\beta,
\end{split}\]
where $\frac{|\Omega|^{2}}{\gamma^{2}}\leq\epsilon $. 
Therefore, taking $\eta$ small enough (note that this implies that $\epsilon$ also is small enough) we have that there exists a sequence $\left\{\omega_{n}\right\}^{\infty}_{n=1}$such that 
\[\partial^{2}_{s}E_{\Omega}(\phi^{s}_{\omega_{n}})|_{s=1}<0, \quad \mbox{for sufficiently large  $n$}.\]
Thus, from Theorem \ref{Ax1} we have that  the standing wave $e^{i\frac{\omega_{n} }{2}t}\phi_{\omega_{n}}(x)$ of \eqref{GP} is unstable. The proof is complete.

\end{proof}





\section{Stability of standing waves}\label{S:5}

This section is devoted to the proof of Theorem \ref{Th2} and Corollary \ref{Et}. The following is the key lemma for our proof.

\begin{lemma}\label{Lxa90}
Let  $|\Omega|<\gamma$ and $1+\frac{4}{N}<p<2^{\ast}$.  For every $r>0$, there exists $q_{0}=q_{0}(r)$, such that if $q<q_{0}$, then 
\begin{equation}\label{Ig}
\inf\left\{E_{\Omega}(u),\quad u\in D_{q}\cap B_{rq/2}\right\}< \inf\left\{E_{\Omega}(u),\quad u\in D_{q}\cap (B_{r}\setminus B_{rq})\right\}.
\end{equation}
\end{lemma}
\begin{proof} Notice that  $-\lambda_{0}>0$. First, we show that $D_{q}\cap B_{r}$ is not empty set iff $q\leq \frac{r}{-\lambda_{0}}$. Indeed,  let $f\in L^{2}(\mathbb{R}^{N})$ be the eigenfunction associated with the eigenvalue $\lambda_{0}$ given in \eqref{impr} such that $\|f\|^{2}_{2}=1$ (the function $f$ can be found in \cite[Section 3]{MATSUNU}). Now we set $\eta(x):=\sqrt{q}f(x)$. For $q\leq \frac{r}{-\lambda_{0}}$,  we see that
\begin{equation*}
\|\eta\|^{2}_{{2}}=q \quad \text{and} \quad \|\eta\|^{2}_{H}=\mathfrak{t}[\eta]=-\lambda_{0}\|\eta\|^{2}_{{2}}\leq r.
\end{equation*}
This implies that $D_{q}\cap B_{r}$ is not-empty. Now, if  $u\in D_{q}\cap B_{r}$, it follows from \eqref{impr},
\begin{equation*}
r\geq \|u\|^{2}_{H}=\mathfrak{t}[u]\geq \lambda_{0}q,
\end{equation*}
that is $q\leq \frac{r}{-\lambda_{0}}$. Next, we show the inequality \eqref{Ig}. A simple computation shows that for  $a>0$ (see proof of Lemma \ref{EEn}),
\begin{equation*}
\frac{1}{2}\mathfrak{t}[u]\geq \left(\frac{1-a}{2}\right)\|\nabla u\|^{2}_{2}+\frac{1}{2}\left(\gamma^{2}-\frac{|\Omega|^{2}}{a}\right)\|xu\|^{2}_{2},
\end{equation*}
Since $|\Omega|<\gamma$, we infer that there exists a constant $C>0$ such that $\|\nabla u\|_{2}\leq C\|u\|_{H}$. By Gagliardo-Nirenberg inequality we have 
\begin{equation*}
\begin{cases} 
E_{\Omega}(u)\geq \frac{1}{2}\| u\|^{2}_{H}-Cq^{\frac{p+1}{2}-\frac{N(p-1)}{4}}\| u\|^{\frac{N(p-1)}{2}}_{H}=\Gamma_{q}(\| u\|_{H}),\\
E_{\Omega}(u)\leq \frac{1}{2}\| u\|^{2}_{H}=\Phi_{q}(\| u\|_{H}),
\end{cases} 
\end{equation*}
where 
\begin{equation*} 
\begin{cases}
\Gamma_{q}(t)=\frac{1}{2}t(1-2Cq^{\chi}t^{\delta})\\
\Phi_{q}(t)=\frac{1}{2}t
\end{cases} 
\end{equation*}
and
\begin{equation*} 
 \chi=\frac{1}{2}\left(p+1-\frac{N(p-1)}{2}\right)>0, \quad \delta=\frac{N(p-1)-4}{4}>0.
\end{equation*}
It is clear that to prove the inequality \eqref{Ig}, we need only show that there exists $0<q_{0}=q_{0}(r)\ll 1$ such that, for every $q<q_{0}$,
\begin{equation*} 
\Phi_{q}(qr/2) <\inf_{t\in (rq,r)}\Gamma_{q}(t).
\end{equation*}
Indeed, it is not difficult to show that there exists $q_{0}>0$, depending only on $r$, $N$ and $p$ such that, if $q<q_{0}$, then $\Gamma_{q}(t)\geq \frac{1}{3}t$ for $t\in (0,r)$. This implies that
\begin{equation*} 
\Phi_{q}(qr/2)=\frac{1}{4}qr<\frac{1}{3}qr\leq\inf_{t\in (rq,r)}\Gamma_{q}(t),
\end{equation*}
and the proof of lemma is complete.
\end{proof}

\begin{proof}[ \bf {Proof of Theorem \ref{Th2}}]
Let $\left\{u_{n}\right\}$ be a minimizing sequence for $J^{r}_{q}$.  Then $\|u_{n}\|^{2}_{2}=q$ and $\|u_{n}\|^{2}_{H}\leq r$. Since $\Sigma\hookrightarrow L^{2}(\mathbb{R}^{N})$ is compact, there exists $\varphi\in \Sigma$, such that $u_{n}\rightharpoonup u$ weakly in $\Sigma$ and $\|\varphi\|^{2}_{2}=q$. Moreover, by the lower semi-continuity we have
\begin{gather*} 
\|\nabla \varphi\|_{2}^{2}+\int_{\mathbb{R}^{N}}V(x)|\varphi(x)|^{2}dx+2l_{\Omega}(\varphi)\\
 \leq\liminf_{n\rightarrow\infty}\left\{ \|\nabla u_{n}\|_{2}^{2}+\int_{\mathbb{R}^{N}}V(x)|u_{n}(x)|^{2}dx+2l_{\Omega}(u_{n}) \right\}.
\end{gather*}
Thus $\varphi\in D_{q}\cap B_{r}$. On the other hand, since $u_{n}\rightarrow \varphi$ in $L^{2}(\mathbb{R}^{N})$, Gagliardo-Nirenberg inequality implies that $u_{n}\rightarrow \varphi$ in $L^{p+1}(\mathbb{R}^{N})$. Again, from  lower semi-continuity we infer $E_{\Omega}(\varphi)\leq \liminf_{n\rightarrow \infty} E_{\Omega}(u_{n})=J^{r}_{q}$. Therefore,  $u\in \mathcal{G}^{r}_{q}$ and $u_{n}\rightarrow \varphi$ in $\Sigma(\mathbb{R}^{N})$. In particular, $l_{\Omega}(u_{n})\rightarrow l_{\Omega}(\varphi)$ as $n\rightarrow\infty$, which completes the proof of (i).

Now we prove (ii).  From \eqref{Ig}, we see that  $\varphi$ does not belong to the boundary of $D_{q}\cap B_{r}$; that is, $E_{\Omega}$ has a local minimum in $\varphi$. We also notice that $\varphi\in B_{rq}$. Therefore, there exists a Lagrange multiplier $\omega\in \mathbb{R}$ such that $\varphi$ satisfies the stationary equation
\begin{equation*} 
-\Delta \varphi+\omega\varphi+V(x)\varphi-2|\varphi|^{p-1}\varphi+2L_{\Omega}\varphi=0.
\end{equation*}
Notice that
\begin{equation*}
J^{r}_{q}\leq E_{\Omega}(\eta)=\frac{1}{2}\mathfrak{t}[\eta]-\frac{2}{p+1}\|\eta\|^{p+1}_{p+1}<-\frac{1}{2}\lambda_{0}q,
\end{equation*}
where  $\eta\in D_{q}\cap B_{r}$ is given in Lemma \ref{Lxa90}.  This implies that
\begin{align}\nonumber
\omega \|\varphi\|^{2}_{2}&=- 2E_{\Omega}(\varphi)+\frac{2(p-1)}{(p+1)}\|\varphi\|^{p+1}_{p+1}      \\\nonumber
&=-2J^{r}_{q}+\frac{2(p-1)}{p+1}\|\varphi\|^{p+1}_{p+1}> \lambda_{0}q.
\end{align}
Since $\|\varphi\|^{2}_{2}=q$, it follows that $\omega>-\lambda_{0}$. Moreover, from $I_{\omega}(\varphi)=0$, we infer
\begin{align*}
\omega \|\varphi\|^{2}_{2}&=-\mathfrak{t}[\varphi]+2\|\varphi\|^{p+1}_{p+1}\\
&\leq -\|\varphi\|^{2}_{H}+C\|\varphi\|^{\frac{N(p-1)}{2}}_{H}q^{\frac{p+1}{2}-\frac{N(p-1)}{4}}\\
&=-\|\varphi\|^{2}_{H}\left(1-C\|\varphi\|^{\frac{N(p-1)}{2}-2}_{H}q^{\frac{p+1}{2}-\frac{N(p-1)}{4}}\right)\\
&\leq -\|\varphi\|^{2}_{H}\left(1-C(rq)^{\frac{N(p-1)}{4}-1}q^{\frac{p+1}{2}-\frac{N(p-1)}{4}}\right).\\
&\leq-\|\varphi\|^{2}_{H}\left(1-Cq^{\frac{p-1}{2}}\right),
\end{align*}
Thus, by using the fact that $\|\varphi\|^{2}_{H}\geq -\lambda_{0}\|\varphi\|^{2}_{2}$ we see that
\begin{equation*}
\omega \leq \lambda_{0}\left(1-Cq^{\frac{p-1}{2}}\right),
\end{equation*}
and finishes the proof.
\end{proof}

Now, we are able to prove the stability the set $\mathcal{G}^{r}_{q}$ given in Corollary \ref{Et}.
\begin{proof}[ \bf {Proof of Corollary \ref{Et}}]
We verify the statement by contradiction. Assume that there exist $\epsilon>0$ and two sequences $\left\{u_{0,n}\right\}\subset \Sigma$ and $\left\{t_{n}\right\}\subset \mathbb{R}$ such that 
\begin{align}\label{L33}
&\inf_{\varphi\in \mathcal{G}^{r}_{q}}\|u_{0, n}-\varphi\|_{\Sigma}<\frac{1}{n} \\\label{L41}
&\inf_{\varphi\in \mathcal{G}^{r}_{q}}\|u_{n}(t_{n})-\varphi\|_{\Sigma}\geq \epsilon\quad \text{ for every $n\in\mathbb{N}$,} 
\end{align}
where $u_{n}(t)$ is the solution to \eqref{GP} with initial datum $u_{0, n}$. A standard argument shows that we can assume $\|u_{0, n}\|^{2}_{2}=q$. The conservation of mass and energy implies that
\begin{align}\label{esc1}
& \|u_{n}(t_{n})\|^{2}_{2}=\|u_{0, n}\|^{2}_{2}=q  \quad \text{ for every $n$,} \\\label{esc2}
& E(u_{n}(t_{n}))=E(u_{0, n})\rightarrow J^{r}_{q} \quad \text{ as $n\rightarrow +\infty$.}
\end{align}
We claim that  there exists a  subsequence $\left\{u_{n_{k}}(t_{n_{k}})\right\}$ of $\left\{u_{n}(t_{n})\right\}$ such that $u_{n_{k}}(t_{n_{k}})\in D_{q}\cap B_{r}$. Indeed, from \eqref{esc1}, we only need to show that $\|u_{n_{k}}(t_{n_{k}})\|^{2}_{H}\leq r$. Suppose, by contradiction, that  there exists $K\geq 1$ such that  $\|u_{n}(t_{n})\|^{2}_{H}> r$ for all $n\geq K$.  By continuity and \eqref{L33}, we infer that there exists $t_{n}^{\ast}\in (0, t_{n})$ such that $\|u_{n}(t^{\ast}_{n})\|^{2}_{H}= r$. Thus, $u_{n}(t^{\ast}_{n})\in D_{q}\cap B_{r}$ and, from \eqref{esc2}, we see that  $\left\{u_{n}(t^{\ast}_{n})\right\}$ is a minimizing sequence of $J^{r}_{q}$. By Theorem \ref{Th2}, there exists $\psi\in\Sigma$ such that $\|\psi\|^{2}_{2}=q$ and $\|\psi\|^{2}_{\Sigma}=r$, which is a contradiction with the fact that the critical point $\psi$ does not belong to the boundary of $D_{q}\cap B_{r}$ (see Lemma \ref{Lxa90}) and the claim follows immediately. On the other hand, using \eqref{esc2}, we see that $\left\{u_{n_{k}}(t_{n_{k}})\right\}$ is  a minimizing sequence for $J^{r}_{q}$. Again, by Theorem \ref{Th2} there exists $f\in \mathcal{G}^{r}_{q}$ such that,  passing to a subsequence if necessary, $\left\{u_{n_{k}}(t_{n_{k}})\right\}$ converges strongly to $f$ in $\Sigma$, which is a contradiction with \eqref{L41}. This completes the proof of Corollary.

\end{proof}

\section*{Acknowledgments} The authors would like to express their sincere thanks to the referees for useful comments
and suggestions that improved the paper.


\end{document}